\documentclass[11pt]{article}
\usepackage{color}
\usepackage{graphicx}
\usepackage{amsmath}
\usepackage{amssymb}
\usepackage{mathrsfs}
\usepackage[numbers,sort&compress]{natbib}
\usepackage{amsthm}
\numberwithin{equation}{section}
\usepackage{pgf}
\usepackage{graphicx}
\usepackage{tikz}
\usepackage{stmaryrd}
\usepackage[latin1]{inputenc}
\usepackage[T1]{fontenc}
\usepackage[english]{babel}
\usepackage{listings}
\usepackage{xcolor,mathrsfs,url}
\usepackage{amssymb}
\usepackage{amsmath}
\usepackage{ifthen}
\usepackage {caption}
\usepackage{xcolor}
\usepackage{cases}

\usepackage{amsmath}
\usepackage{amsthm}
\usepackage{todonotes}
\usepackage{float}
\usepackage{subfigure}
\usepackage{hyperref}
\hypersetup{
    colorlinks=true, %set true if you want colored links
    linktoc=all,     %set to all if you want both sections and subsections linked
    linkcolor=blue,  %choose some color if you want links to stand out
    citecolor=blue%choose some color if you want links to stand out
}

\def\R{\mathbb R}
\def\C{\mathbb C}

\def\P{\mathcal P}
\def\d{\rm d}
\def\H{\mathcal H}

\newtheorem{theorem}{Theorem}[section]
\newtheorem{lemma}{Lemma}[section]
\newtheorem{proposition}{Proposition}[section]

\begin{document}

\title{Existence of global solutions to the nonlocal Schr\"{o}dinger equation on the line }
\author{Yi ZHAO$^1$ and Engui FAN$^{1}$\thanks{\ Corresponding author and email address: faneg@fudan.edu.cn } }
\footnotetext[1]{ \  School of Mathematical Sciences, Fudan University, Shanghai 200433, P.R. China. }

\date{ }

\maketitle
\begin{abstract}
\baselineskip=18pt
In this paper, we  address   the existence of global solutions to the Cauchy problem for the integrable nonlocal nonlinear Schr\"{o}dinger (nonlocal NLS)
equation with the initial data $q_0(x)\in H^{1,1}(\R)$ with the $L^1(\R)$  small-norm assumption. 
 We rigorously  show that  the spectral problem for the nonlocal NLS equation  admits no eigenvalues or resonances, as well as Zhou vanishing lemma 
 is effective  under the $L^1(\R)$  small-norm assumption. 
  With inverse scattering theory and the Riemann-Hilbert approach,
 we rigorously establish  the bijectivity  and Lipschitz continuous  of the direct and inverse scattering map  from the initial data to reflection coefficients.
By using reconstruction formula and the   Plemelj  projection  estimates of reflection coefficients,
we  further obtain  the existence of the local solution and the  priori estimates, which assure
   the existence of the global solution to the Cauchy problem for the nonlocal NLS equation.\\[6pt]
\\[6pt]

\end{abstract}

\baselineskip=18pt

\newpage

\tableofcontents

\section{ Introduction}

In this paper, we show  the existence of global solutions to the Cauchy problem for the integrable nonlocal nonlinear Schr\"{o}dinger (nonlocal NLS) equation
\begin{align}
&iq_t(x,t)+q_{xx}(x,t)+2\sigma q^2(x,t)\bar{q}(-x,t)=0,\ \sigma=\pm 1,\  x\in\R,\ t>0\label{nnls}\\
&q(x,0)=q_0(x),\label{initial}
\end{align}
where the subscripts denote partial derivatives and the bar denotes complex conjugation throughout the article, and $q(x,t)$ is a complex valued function of the real variables $x$ and $t$.

%Nonlinear evolution equations are an important class of equations to describe the nonlinear wave propagation. And the integrable equations play an physically important role among the nonlinear evolution equations since many  integrable equations can be derived from basic principles and arise as universal models in many physical phenomenon. For example, the Korteweg-de Vries(KdV) equation describe the evolution of weakly dispersive and small amplitude waves in quadratic media. It has  an important application to shallow waves \cite{1,2}.  Another universal model is the integrable cubic nonlinear Schr\"{o}dinger equation. It is a key model describing optical wave propagation in Kerr media in the context of nonlinear optics. And it describes the evolution of weakly nonlinear and quasi-momochromatic wave trains in media with cubic nonlinearities \cite{3,4}.

 The nonlocal  NLS  equation (\ref{nnls})  was first introduced as an  integrable model  by Ablowitz and Musslimani in 2013 \cite{5}, and
  they further obtain its  Lax pair,  an infinite number of conservation laws and is  PT  symmetry.
 For    rapidly decaying initial data and  initial value with  nonzero boundary conditions, the  soliton solutions to the nonlocal  NLS   equation (\ref{nnls}) were obtained via the inverse scattering transform , respectively \cite{5,16}.
 In fact,  the nonlocal  NLS   equation ever  was   derived in a physical application of magnetics  \cite{10}.
 The  nonlocal  NLS  equation  (\ref{nnls}) also  can  be regarded  as a linear Schr\"{o}dinger equation
\begin{equation}
    iq_t(x;t)=q_{xx}(x,t)+V[q,x,t]q(x,t)
\end{equation}
with a self-induced potential $V[q,x,t]=V^*[q,-x,t]$, thus the  nonlocal  NLS   equation  (\ref{nnls}) is a PT symmetric equation \cite{6,9}.
Since PT symmetric systems allow for loseless-like propagation due to their balance of gain and loss \cite{7},
 they have attracted considerable attention in recent years.  The PT symmetric system is a key model on linear and nonlinear waves \cite{8} and related to the cutting edge research area of modern physics \cite{7,9}.
Besides, the nonlocal  NLS  equation is gauge-equivalent to  the unconventional system of coupled Landau-Lifchitz equations and therefore can be useful in the physics of nanomagnetic
 artificial materials \cite{8}. Possible application of the nonlocal NLS equation  is discussed in the context of Alice-Bob systems \cite{11,12}.

Related to the nonlocal  NLS  equation, Fokas analyzed a (2+1) dimensional integrable nonlocal NLS equation \cite{13}.
 Ablowitz introduced new reverse space-time and reverse time nonlocal nonlinear integrable
equations and their discrete version \cite{14}. They identified new nonlocal symmetry reductions for the general AKNS system and addressed the scattering problem.  Besides, an integrable discrete PT symmetric "discretization" of the nonlocal NLS equation was obtained from a new nonlocal PT symmetric reduction of the Ablowitz-Ladik scattering problem \cite{15}.  Nonlocal versions of some other integrable equations such as  the modified KdV equation and sine-Gordan equation were investigated \cite{14}.

 The long-time behaviour of the solution to the nonlocal NLS equation (\ref{nnls}) with decaying boundary conditions was investigated via the nonlinear
 Deift-Zhou steepest-decent method \cite{Rybalko1}.
  Recently,   the long-time behaviour for  nonlocal NLS equation (\ref{nnls}) with step-like initial data
 was obtained \cite{Rybalko2,Rybalko3}.
  Comparing with the classical NLS equation \cite{NLS1,NLS2,NLS3,NLS4,NLS5,NLS6,NLS7,NLS8},
the   nonlocal  NLS   equation (\ref{nnls})  displays   some  different characteristics and
 interesting properties   both  on  exact solutions and long time asymptotic  behavior. For weighted Sobolev initial data $q_0(x)\in H^{1,1}(\mathbb{R})$,
 we obtained  long time asymptotic  behavior for  the nonlocal NLS  equation (\ref{nnls}) in  solitonic region \cite{LYF}. However,
the global existence of  the   nonlocal  NLS equation (\ref{nnls})   still  has been    unknown.
 It is well-known that the existence of global solution or well-posedness of initial value problem of a partial
 differential equation is the theoretical guarantee to the long time asymptotic analysis. In general,   for  the initial data
in  a Sobolev space,  the solution obtained through the inverse scattering method  exists  in a larger Sobolev space.
Zhou ever established  $L^2$-Sobolev space bijectivity for the
scattering-inverse scattering transforms associated with the ZS-AKNS system \cite{zhouxin}.  Recently Pelinovsky and  Shimabukuro
established  the  existence of global solutions for
the Cauchy problem for derivative NLS equation with inverse scattering method \cite{18}.

In our paper, we  try to   rigorously establish   bijectivity for the
scattering-inverse scattering transforms associated with  nonlocal NLS equation (\ref{nnls}), and  further
show   the global existence for  the initial value problem    (\ref{nnls})-(\ref{initial})
with inverse scattering transform  and Riemann-Hilbert
   (RH)  method.
  Starting from the given initial data, the direct transform gives rise to the scattering data. Then, the inverse scattering transform   goes  back  to the solution to the original Cauchy problem  (\ref{nnls})-(\ref{initial}) based on the associated RH problem.

The structure of the paper is as follows. In Section \ref{sec:section2}, we present the direct scattering
 transform to the initial value problem  (\ref{nnls})-(\ref{initial})  based on its Lax pair.  The analytical, asymptotic and integrability for
  the Jost functions and  the scattering coefficients are analysed in details.  We establish the Lipschitz continuous mapping from the initial data to the reflection coefficients.
 In Section \ref{sec:section3}, we carry out the inverse scattering transform. We set up the RH problem and prove its existence and uniqueness via a general vanishing lemma under the $L^1(\R)$  small-norm. In Section \ref{sec:section4}, we reconstruct  and estimate the potential associated with the solutions of the RH problem and reflection coefficients.
 Further, we establish the Lipschitz continuous mapping from the reflection coefficients to the potential. In Section \ref{sec:section5}, we perform the time evolution
 od reflection coefficients and RH problem  and further prove the existence of the local solution and global solution to the initial value problem  (\ref{nnls})-(\ref{initial}).
\section{ Direct scattering transform}
\label{sec:section2}
\subsection{Some notations}
We first fix some notations used in this paper:
%Further, we define weighted Sobolev space by
\begin{itemize}
\item{} If $I$ is an interval on the real line $\mathbb{R}$ and $X$ is a  Banach space, then $C (I,X)$ denotes the space of continuous functions on $I$ taking values in $X$. It is equipped with the norm
\begin{equation*}
	\|f\|_{C (I, X)}=\sup _{x \in I}\|f(x)\|_{X}.
\end{equation*}
\item{} $H^m(\R)$ denotes the Sobolev space of distributions with square integrable derivatives up to the order m.
\item{} $L^{2,m}(\R)$ denotes the weighted $L^2(\R)$ space with the norm
$$\| q\|_{L^{2,m}(\R)} =\left(\int_{\R} \langle x\rangle^{2m}|q(x)|^2\d x \right)^{1/2}$$
where $\langle x\rangle=(1+x^2)^{1/2}$.
\item{} $H^{1,1}(\R)$ denotes the weighted Sobolev space
    $$H^{1,1}(\R)=\left\{q|q\in L^{2,1}(\R), \partial_x q\in L^{2,1}(\R)\right\}.$$
\item{} $\mathcal H(\R)$ denotes the function space
$$\H(\R)=\left\{f(x)|f(x)\in H^1(\R)\cap L^{2,1}(\R),xf(x)\in L^{\infty}(\R)\right\}$$
with the norm
$$\|f(x) \|_{\H(\R)}=\| f(x)\|_{H^1(\R)\cap L^{2,1}(\R)}+\| xf(x)\|_{L^{\infty}(\R)}.$$
\end{itemize}

\subsection{Jost functions and Lipschitz continuous}

The nonlocal NLS equation is integrable and admits the Lax pair\cite{5}
\begin{align}
    &\Phi_x+ik\sigma_3\Phi=U\Phi,\label{spectral}\\
    &\Phi_t+2ik^2\sigma_3\Phi=V\Phi,\label{3}
\end{align}
where the function  $\Phi(x,t;k)$ is a matrix-valued function, $k$ is a spectral parameter, and the matrices U and V are given by
\begin{equation}
U=\left(\begin{array}{cc}
 0 & q(x,t)\\
 -\sigma\bar{q}(-x,t) & 0\end{array}\right),\quad V=\left(\begin{array}{cc}
 A & B\\
 C & -A\end{array}\right),
\nonumber
\end{equation}
with
\begin{equation}
\begin{split}
    &A=i\sigma q(x,t)\bar{q}(-x,t), \\
    &B=2kq(x,t)+iq_x(x,t),\\
    &C=-2k\sigma\bar{q}(-x,t)+i\sigma(\bar{q}(-x,t))_x.
\end{split}
\nonumber
\end{equation}
The matrix $\sigma_3$ is the standard Pauli matrix
$$\sigma_3=\left( \begin{array}{cc}
1 & 0\\
0 & -1
\end{array}\right).$$

By a normalized transformation
\begin{equation}
    \varphi_{\pm}(x;k)=[\Phi]_1(x;k)e^{ikx},\quad \phi_{\pm}(x;k)=[\Phi]_2(x;k)e^{-ikx},
    \nonumber
\end{equation}
where $[\Phi]_1$ and $[\Phi]_2$ denote the first  and second column of the matrix $\Phi$, we have
\begin{equation}
\begin{split}
    &\varphi_{\pm}(x;k)\rightarrow e_1\quad as \quad x\rightarrow\pm\infty, \\
    &\phi_{\pm}(x;k)\rightarrow e_2\quad as \quad x\rightarrow\pm\infty,\\
\end{split}
   \label{xasy}
\end{equation}
and the Jost functions $\varphi_{\pm}(x;k)$ and $\phi_{\pm}(x;k)$ satisfy the Volttera's integral equation
    \begin{align}
    &\varphi_{\pm}(x;k)=e_1+\int_{\pm\infty}^x {\rm diag}(1,e^{2ik(x-y)})
 %   \left(\begin{array}{cc}
%    1 & \ \\
%    \  & e^{2ik(x-y)}
%    \end{array}\right)
    U(q(y))\varphi_{\pm}(y;k)dy,\label{5}\\
     &\phi_{\pm}(x;k)=e_2+\int_{\pm\infty}^x {\rm diag}(e^{-2ik(x-y)},1)
   %  \left(\begin{array}{cc}
%    e^{-2ik(x-y)} & \ \\
%    \  & 1
%    \end{array}\right)
    U(q(y))\phi_{\pm}(y;k)dy.\label{6}
    \end{align}

\begin{lemma}\label{2.1}  If  $q\in L^{1}(\R)$ and $\| q\|_{L^1(\R)}<1$, then or every $k\in\R$, there exist unique solutions $\varphi_{\pm}(\cdot;k)\in L^{\infty}(\R)$ and $\phi_{\pm}(\cdot;k)\in L^{\infty}(\R)$ of the integral equations (\ref{5}) and (\ref{6}), respectively. Moreover, for every $x\in\R$, $\varphi_-(x;\cdot)$ and $\phi_+(x;\cdot)$ are continued analytically in $\C^+$, whereas $\varphi_+(x;\cdot)$ and $\phi_-(x;\cdot)$ are continued analytically in $\C^-$.
\end{lemma}
\begin{proof}
We only give the proof for the Jost function $\varphi_-$, others can be given by similar procedure.

For a   vector function  $f(x;k)=(f_1(x;k), f_2(x;k))^T$, define a operator $K$ by
\begin{equation}
(Kf)(x;k)=\int_{-\infty}^x {\rm diag}(1,e^{2ik(x-y)})U(q(y))f(y;k)dy,
\label{31}\end{equation}
then  the integral equation can be rewritten as
\begin{equation}
    \varphi_-=e_1+K\varphi_-,
\end{equation}
its deformation is
$$(I-K)\varphi_-=e_1.$$
According to the Fredholm's alternative theorem, it is sufficient to prove that there exists a unique zero solution to the homogeneous  equation $Kf=f$.

Since $q\in L^1(\R)$, for every $k\in\C^+$, we have
\begin{equation}
    \| Kf(\cdot;k)\|_{L^{\infty}(\R)}\leq \| q\|_{L^1(\R)}\|f(\cdot;k) \|_{L^{\infty}(\R)}.
    \label{LxLk}
\end{equation}
where the norm is defined by
$$\| f(\cdot;k)\|_{L^{\infty}(\R)}=\| f_1(\cdot;k)\|_{L^{\infty}(\R)}+f_2 (\cdot;k)\|_{L^{\infty}(\R)}.$$
Due to the small-norm assumption $\| q\|_{L^1(\R)}<1$,  the operator $K$ is a contraction from $L^{\infty}(\R)$ to $L^{\infty}(\R)$.  By the Banach fixed point theorem, there exists a unique solution $\varphi_-(\cdot;k)\in L^{\infty}(\R)$ of the integral equation (\ref{31}) for  every $k\in\C^+$.

%In order to extend the results of $(-\infty,x_0)$ to the whole line $\R$, we spilt $\R$ into a finite number of subintervals satisfying the estimate $\|q\|_{L^1}<1$, then glue them together and obtain the unique solution $\varphi(\cdot;k)\in L^{\infty}(\R)$ for every $k\in\C^+$.

To prove the analyticity of $\varphi(x;\cdot)$ in $\C^+$, we define Neumann series
$$w(x;k)=\sum_{n=0}^{\infty}w_n(x;k)$$
with
\begin{align}
&w_0=e_1, \ \ w_{n+1}(x;k)=\int_{-\infty}^x {\rm diag}(1,e^{2ik(x-y)})U(q(y))w_n(y;k)dy.\nonumber
\end{align}
We see that
$$\|w_{n+1}(x;k)\|_{L^{\infty}_x}\leq \frac{1}{n!}\| U(q)\|^n_{L^1},$$
consequently, the Neumann series $w(x;k)$ converges absolutely and uniformly for every $x\in\R$ and $k\in\C^+$. As a result, $\varphi_-(x;\cdot)$ is analytic in $\C^+$ for every $x\in\R$.
\end{proof}
\begin{lemma} \label{2.2}Let $q\in L^1(\R)$ and $\| q\|_{L^1(\R)}<1$. For every $x\in\R$, the Jost functions $\varphi_{\pm}$ and $\phi_{\pm}$ satisfy the following limits as |k| approaches to infinity in their analytic domains such that $|\rm{Im}\ k|\rightarrow\infty$
    \begin{align}
        &\lim_{|k|\rightarrow \infty} \varphi_{\pm}(x;k)= e_1,\label{lim1}\\
        &\lim_{|k|\rightarrow \infty} \phi_{\pm}(x;k)= e_2.
    \end{align}
If in addition, $q\in C(\R)$, then
    \begin{align}
        &\lim_{|k|\rightarrow \infty} 2i\sigma k[\varphi_{\pm}(x;k)-e_1]=s_{\pm}^{(1)}(x)e_1+\bar{q}(-x)e_2,\label{9}\\
         &\lim_{|k|\rightarrow \infty} 2i\sigma k[\phi_{\pm}(x;k)-e_2]=-\sigma{q}(x)e_1+s_{\pm}^{(2)}(x)e_2,\label{10}
         \end{align}
where
\begin{equation}
\begin{split}
&s_{\pm}^{(1)}(x)=\int_{\pm\infty}^x q(y)\bar{q}(-y) \d y,\\
%&q_{\pm}^{(2)}(x)=\frac{\sigma}{2i}\bar{q}(-x),\\
%&s_{\pm}^{(1)}(x)=-\frac{1}{2i}{q}(x),\\
&s_{\pm}^{(2)}(x)=\int_{\pm\infty}^x q(y)\bar q(-y) \d y.
\end{split}\label{q12}
\end{equation}
\end{lemma}
\begin{proof}
Again, we only give the proof of the Jost function $\varphi_-$. Let $\varphi_-=[\varphi_-^{(1)}, \varphi_-^{(2)}]^t$. It follows from (\ref{5}) and (\ref{6}) that
\begin{align}
&\varphi_-^{(1)}(x;k)=1+\int_{-\infty}^x q(y)\varphi_-^{(2)}(y;k)dy,\label{112}\\
&\varphi_-^{(2)}(x;k)=-\sigma\int_{-\infty}^x e^{2ik(x-y)}\bar{q}(-y)\varphi_-^{(1)}(y;k)dy.\label{113}
\end{align}
 We acquire  $\varphi_-(\cdot;k)\in L^{\infty}(\R)$ thanks to Lemma \ref{2.1}. Employing  $q\in L^1(\R)$,
 we have
\begin{equation}
    \lim_{|k|\rightarrow\infty}\varphi_-^{(2)}(x;k)=0
\end{equation}
by Lebesgue's dominated convergence theorem. Similarly, we arrive at
$$\lim_{|k|\rightarrow\infty}\varphi_-^{(1)}(x;k)=1$$
in view of  $q\in L^1(\R)$ and $\varphi_-(\cdot;k)\in L^{\infty}(\R)$. This completes the proof of (\ref{lim1}) for $\varphi_-$.

If in addition $q\in C(\R)$, then for small $\delta >0$, we rewrite (\ref{113}) as
\begin{equation}
    \begin{split}
        \varphi_-^{(2)}(x;k)=&\int_{-\infty}^{x-\delta} e^{2ik(x-y)}\eta(y;k)dy+\eta(x;k)\int_{x-\delta}^x e^{2ik(x-y)}dy
        \\&+\int_{x-\delta}^x e^{2ik(x-y)}[\eta(y;k)-\eta(x;k)]dy\\
        =&I_1+I_2+I_3,
    \end{split}
\end{equation}
where $\eta(x;k)=-\sigma\bar{q}(-x)\varphi_-^{(1)}(x;k)$.

Since $\eta(\cdot;k)\in L^1(\R)$, we get
$$|I_1|\leq e^{-2\delta\rm{Im \ k}}\|\eta(\cdot;k) \|_{L^1(\R)}.$$
Since $\eta(\cdot;k)\in C(\rm{R})$, we give
\begin{equation}
\begin{split}
|I_3| \leq& \|\eta(\cdot;k)-\eta(x;k) \|_{L^{\infty}(x-\delta,x)}\left| \int_{x-\delta}^x e^{2ik(x-y)} \d y \right|\\
\leq& \frac{1}{2\rm{Im\ k}}\|\eta(\cdot;k)-\eta(x;k) \|_{L^{\infty}(x-\delta,x)},
\nonumber
\end{split}
\end{equation}
and
$$I_2=\frac{1}{2ik}(e^{2ik\delta}-1)\eta(x;k).$$
Choose $\delta={\rm{Im}\ k}^{-\frac 12}$ such that $\delta\rightarrow 0$ as $|k|\rightarrow\infty$. Thus $kI_1$ and $kI_3$ approach to zero as $|k|\rightarrow\infty$. Therefore, we obtain \begin{equation}\lim_{|k|\rightarrow\infty}k\varphi_-^{(2)}(x;k)=\lim_{|k|\rightarrow\infty}kI_2  =\frac{\sigma}{2i}\bar{q}(-x).
\label{15}
\end{equation}
Multiplying (\ref{112}) with k and substituting (\ref{15}) leads to
\begin{equation}
    \lim_{|k|\rightarrow\infty} k(\varphi_-^{(1)}-1)=\frac{\sigma}{2i}\int_{-\infty}^x q(y)\bar{q}(-y) \d y.
\end{equation}
This completes the proof of (\ref{9}) for $\varphi_-$.
\end{proof}
\begin{proposition}(see \cite{18}) If $w\in L^2(\R)$, then
\begin{equation}
    \sup_{x\in\rm{R}}\left\| \int_{-\infty}^x e^{2ik(x-y)}w(y)dy \right\|_{L^2_k(\R)}\leq \sqrt{\pi}\| w\|_{L^2(\R)}.
\label{17}
\end{equation}
And if $w\in H^1(\R)$, then
\begin{equation}
    \sup_{x\in\rm{R}}\left\| 2ik\int_{-\infty}^x e^{2ik(x-y)}w(y)dy+w(x) \right\|_{L^2_k(\R)}\leq \sqrt{\pi}\| \partial_x w\|_{L^2(\R)}.
\label{18}
\end{equation}
Moreover, if $w\in L^{2,1}(\R)$, then for every $x_0\in \R^-$, we have
\begin{equation}
    \sup_{x\in (-\infty, x_0)}\left\|\langle x \rangle \int_{-\infty}^x e^{2ik(x-y)}w(y)dy \right\|_{L^2_k(\R)}\leq \sqrt{\pi}\| w\|_{L^{2,1}(-\infty, x_0)},
\label{19}
\end{equation}
where $\langle x \rangle=(1+x^2)^{\frac 1 2}$.
\end{proposition}
\begin{lemma} If $q\in L^{2,1}(\R)$ and $\| q\|_{L^1(\R)}< 1$, then for every $x\in\R^{\pm}$, we have
\begin{equation}
    \varphi_{\pm}(x;\cdot)-e_1\in H^1(\R),\quad \phi_{\pm}(x;\cdot)-e_2\in H^1(\R).
\label{20}
\end{equation}
Moreover, if $q\in H^1(\R)\cap L^{2,1}(\R)$, for every $x\in\R$, we have
\begin{align}
    & 2i\sigma k(\varphi_{\pm}(x;k)-e_1)-(s_{\pm}^{(1)}(x)e_1+\bar q(-x)e_2)\in L^2_k(\R),\label{21}\\
     & 2i\sigma k(\phi_{\pm}(x;k)-e_2)-(-\sigma q(x)e_1+s_{\pm}^{(2)}(x)e_2)\in L^2_k(\R)\label{kphi}.
\end{align}
\end{lemma}
\begin{proof}
Again, we only give the proof of $\varphi_-$. Recall $\varphi_-=e_1+K\varphi_-$ where the operator K is the same as the one in Lemma 1. Subtracting $( {I}-K)e_1$
from the equation $({I}-K)\varphi_-=e_1$ and calculating leads to
\begin{equation}\begin{split}
    ({I}-K)(\varphi_--e_1)=& e_1-(I-K)e_1=he_2,
\end{split}\label{23}\end{equation}
where $h(x;k)=-\sigma\int_{-\infty}^x e^{2ik(x-y)}\bar{q}(-y)\d y$.

It is sufficient to prove (\ref{20}) for $\varphi_-$ that the operator $\rm{I}-K$ is invertible and bounded as well as   $h(x;k)\in L_x^{\infty}(\R; L^2_k(\R))$ and $\partial_k \varphi_-(x;k)\in L_x^{\infty}(\R;L^2_k(\R))$.

Recall the formulation of the  operator K in the  Lemma \ref{2.1}, we have
\begin{equation}
    \| K^n f(x;k) \|_{L_x^{\infty}L^2_k}\leq \frac{1}{n!} \|U(q) \|_{L^1}^n \|f(x;k) \|_{L_x^{\infty}L^2_k},
\end{equation}
where $\|f\|_{L_x^{\infty}L_k^2}=\sup_{x\in\R}\|f\|_{L^2_k(\R)}$ and $\| U(q)\|_{L^1}=\sum_{i,j=1}^2\| U_{ij}\|_{L^1}$,
thus we arrive at $\| K^n f(x;k) \|_{L_x^{\infty}L^2_k}\rightarrow 0$
which means the operator ${I}-K$ is invertible and bounded due to
\begin{equation}
    \|({I}-K)^{-1} \|\leq \sum_{n=0}^{\infty}\frac{1}{n!}\| U(q)\|_{L_1}^n=e^{\| U(q)\|_{L^1}}.
\label{25}\end{equation}

Since $q\in L^{2,1}(\R)$, it follows from (\ref{17}) in proposition 1 that
\begin{equation} \sup_{x\in\R}\| h(x;k)  \|_{L^2_k(\R)}\leq \sqrt{\pi} \| q\|_{L^2(\R)},
\end{equation}
then we arrive at $h(x;k)\in L_x^{\infty}(\R; L^2_k(\R))$
and for every $x_0\in\rm{R}^-$,
 \begin{equation}
 \sup_{x\in(-\infty, x_0)}\| \langle x \rangle h(x;k)\|_{L^2_k(\R)}\leq \sqrt{\pi}\| q\|_{L^{2,1}(-\infty, x_0)},
\label{27} \end{equation}
  Combining (\ref{25}) with (\ref{27}), we acquire
 \begin{equation}
     \sup_{x\in(-\infty, x_0)}\| \langle x \rangle( \varphi_-(x;k)-e_1 )\|_{L^2_k(\R)}\leq \sqrt{\pi}e^{\| U(q)\|_{L^1}}\| q\|_{L^{2,1}(\R)}.
\label{28} \end{equation}
 Next we consider the vector ${v}(x;k)$. Let
 $${v}(x;k)=\left(\begin{array}{c}
 \partial_k \varphi_-^{(1)}(x;k)\\
 \partial_k \varphi_-^{(2)}(x;k)-2ix\varphi_-^{(2)}(x;k)
 \end{array}\right).$$
 Direct calculation yields
 \begin{equation}
     ({I}-K){v}=\left(\begin{array}{c}\partial_k\varphi_-^{(1)}-\int_{-\infty}^x q(y)(\partial_k\varphi_-^{(2)}-2iy\varphi_-^{(2)})\d y
 \\
 \partial_k\varphi_-^{(2)}-2ix\varphi_-^{(2)}+\int_{-\infty}^x \sigma e^{2ik(x-y)}\bar q(-y)\partial_k \varphi_-^{(1)}(y;k) \d y
 \end{array}\right).
\nonumber\end{equation}
Taking derivative with k for the integral equation (\ref{5}) gives
\begin{equation}
    \begin{split}
         &\partial_k \varphi_-^{(1)}=\int_{-\infty}^x q(y)\partial_k\varphi_-^{(2)}(y;k)\d y,\\
         &\partial_k \varphi_-^{(2)}=-\sigma \int_{-\infty}^x e^{2ik(x-y)}\bar q(-y)\partial_k \varphi_-^{(1)}(y;k)\d y-2i\sigma\int_{-\infty}^x(x-y)e^{2ik(x-y)}\bar q(-y)\varphi_-^{(1)}(y;k)\d y,
    \end{split}\nonumber
\end{equation}
and  utilizing the integral equation (\ref{5}) of $\varphi_-^{(2)}$,  we derive that
\begin{equation}
    (I-K)v=h_1e_1+(h_2+h_3)e_2,
\label{331}\end{equation}
with
\begin{equation}
    \begin{split}
        &h_1(x;k)=2i\int_{-\infty}^x yq(y)\varphi_-^{(2)}(y;k)\d y,  \\
        &h_2(x;k)=2i\sigma \int_{-\infty}^x ye^{2ik(x-y)}\bar{q}(-y)(\varphi_-^{(1)}(y;k)-1) \d y,   \\
        &h_3(x;k)=2i\sigma \int_{-\infty}^x ye^{2ik(x-y)}\bar{q}(-y) \d y.
    \end{split}
\end{equation}
 We estimate the first term in (\ref{331}) by the Minkowski's inequality of integral form and find that for every $x_0\in\R^-$,
 \begin{equation}
     \begin{split}
         \|h_1(x;k) \|_{L^2_k(\R)}\leq & \int_{-\infty}^x \|2iyq(y)\varphi_-^{(2)}(y;k) \|_{L^2_k(\R)} \d y\\
         \leq &c\| q\|_{L^1(\R)}\sup_{x\in(-\infty,x_0)} \|\langle x\rangle \varphi_-^{(2)}(x;k) \|_{L^2_k(\R)}.
     \end{split}
 \nonumber\end{equation}
 where c is a constant.

 Due to the imbedding of $L^{2,1}(\R)$ into $L^1(\R)$ and (\ref{28}), we infer that
 \begin{equation}
     \sup_{x\in(-\infty,x_0)} \|h_1(x;k) \|_{L^2_k(\R)}
    % \leq  c\| q\|_{L^1(\R)}\sup_{x\in(-\infty,x_0)} \|\langle x\rangle \varphi_-^{(2)}(x;k) \|_{L^2_k(\R)}
     < +\infty,
 \end{equation}
Performing a similar analysis yields
 \begin{equation}
 \begin{split}
     \|h_2(x;k) \|_{L^2_k(\R)}
     \leq & \int_{-\infty}^x  \|2i\sigma ye^{2ik(x-y)}\bar{q}(-y)(\varphi_-^{(1)}(y;k)-1)\|_{L^2_k(\rm R)} \d y\\
     \leq &  c\| q\|_{L^1(\R)}\sup_{x\in(-\infty,x_0)} \|\langle x \rangle (\varphi_-^{(1)}(x;k))-1 \|_{L^2_k(\R)}\\
     <&+\infty,
 \end{split}
 \nonumber\end{equation}
 %and
% \begin{equation}
%     \sup_{x\in(-\infty,x_0)} \|h_2(x;k) \|_{L^2_k(\R)}\leq  c\| q\|_{L^1} \sup_{x\in(-\infty,x_0)} \|\langle x\rangle (\varphi_-^{(1)}(x;k)-1) \|_{L^2_k(\R)}<+\infty.
% \end{equation}
For the second term of (\ref{331}), we utilize the estimate (\ref{19}) of proposition 1 and find that
 \begin{equation}
     \sup_{x\in(-\infty,x_0)} \|h_3(x;k) \|_{L^2_k(\R)}\leq c \|q \|_{L^{2,1}(\R)}<+\infty.
 \end{equation}
 As a result,  we conclude that $v(x;k)\in L_x^{\infty}((-\infty, x_0); L^2_k(\R))$.

 Since $x\varphi_-^{(2)}(x;k)\in L_x^{\infty}((-\infty, x_0); L^2_k(\rm R))$ given the estimate (\ref{28}), we obtain  $\partial_k \varphi_-(x;k)\in L_x^{\infty}((-\infty, x_0); L^2_k(\R))$. This completes the proof of (\ref{20}) for $\varphi_-$.

Moreover, if $q\in H^1(\R)\cap L^{2,1}(\R)$, using the operator $I-K$ again leads to
\begin{equation}\begin{split}
    &( I-K)(k(\varphi_--e_1)-(s_-^{(1)}e_1+\frac{\sigma}{2i}\bar q(-x)e_2))\\
    &=khe_2-(I-K)(s_-^{(1)}e_1+\frac{\sigma}{2i}\bar q(-x)e_2)\\
    &=\left(\begin{array}{c}
    0\\%-\frac{\sigma}{2i}\int_{-\infty}^x q(y)\bar q(-y)\d y+\int_{-\infty}^x\frac{\sigma}{2i}q(y)\bar q(-y)\d y\\
    -\sigma k\int_{-\infty}^xe^{2ik(x-y)}\bar q(-y)\d y-\frac{\sigma}{2i}\bar q(-x)-\sigma\int_{-\infty}^x e^{2ik(x-y)}\bar q(-y)s_-^{(1)}(y)\d y
    \end{array}\right).
\end{split}\nonumber\end{equation}
Let
\begin{equation}
    \tilde h(x;k)=-\frac{\sigma}{2i}\left( 2ik\int_{-\infty}^xe^{2ik(x-y)}\bar q(-y)\d y+\bar q(-x)\right)-\sigma\int_{-\infty}^x e^{2ik(x-y)}\bar q(-y)s_-^{(1)}(y)\d y,
\nonumber\end{equation}
we have
\begin{equation}
    ( I-K)(k(\varphi_--e_1)-(s_-^{(1)}e_1+\frac{\sigma}{2i}\bar q(-x)e_2))=\tilde h e_2.
\label{H1}\end{equation}
On account of the estimate (\ref{18}) of  proposition 1, we infer that
\begin{equation}
        \sup_{x\in\R} \left\|\left( 2ik\int_{-\infty}^xe^{2ik(x-y)}\bar q(-y)\d y+\bar q(-x)\right) \right\|_{L^2_k(\R)}\\
        \leq c\|\partial_x q \|_{L^2(\R)}<+\infty,
\label{tildeh1}
\end{equation}
As $q\in L^2(\R)$, we have
$\bar q(-x) s_-^{(1)}(x)=\frac{\sigma}{2i}\bar q(-x)\int_{-\infty}^xq(y)\bar q(-y)\d y\in L^2(\R).$
 Then according to the  estimate (\ref{17}), we have
\begin{equation}
    \sup_{x\in\R} \left\|\int_{-\infty}^x e^{2ik(x-y)}\bar q(-y)s_-^{(1)}(y)\d y \right\|_{L^2_k(\R)}\leq c\|\bar q(-x) s_-^{(1)}(x)\|_{L^2(\R)}<+\infty.
\label{tildeh2}\end{equation}
Then it follows from (\ref{tildeh1}) and (\ref{tildeh2}) that $\tilde h(x;k)\in L_x^{\infty}(\R; L^2_k(\R))$.
Cosequently, We  conclude from (\ref{H1}) that for every $x\in\R$,  $k(\varphi_-(x;k)-e_1)-(s_-^{(1)}(x)e_1+\frac{\sigma}{2i}\bar q(-x)e_2)$ belongs to $L^2_k(\R)$ because the operator $I-K$ is invertible and bounded.
\end{proof}

\begin{lemma} Let  $q\in H^1(\R)\cap L^{2,1}(\R)$ and   $\| q\|_{L^1(\R)}<1 $.
The mappings
\begin{align}
   &  L^{2,1}(\R)\ni q  \mapsto [\varphi_{\pm}(x;k)-e_1, \phi_{\pm}(x;k)-e_2]
    \in L_x^{\infty}(\R^{\pm}; H^1_k(\R)),\\
   & H^1(\R)\cap L^{2,1}(\R)\ni q  \mapsto
    [\hat{\varphi}_{\pm}, \hat{\phi}_{\pm}]
   \in L_x^{\infty}(\R; L^2_k(\R))
\end{align}
are Lipschitz continuous,
here
\begin{equation}
    \begin{split}
        &\hat{\varphi}_{\pm}=2i\sigma k(\varphi_{\pm}(x;k)-e_1)-(s_{\pm}^{(1)}(x)e_1+\bar q(-x)e_2), \\
        &\hat{\phi}_{\pm}=2i\sigma k(\phi_{\pm}(x;k)-e_2)-(-\sigma q(x)e_1+s_{\pm}^{(2)}(x)e_2).
    \end{split}\nonumber
\end{equation}
\end{lemma}
\begin{proof}
Let $q, \tilde{q}\in L^{2,1}(\R)$. Let the  functions $[\varphi_{\pm}, \phi_{\pm}]$ and $[\tilde{\varphi}_{\pm}, \tilde{\phi}_{\pm} ]$ denote the corresponding Jost functions,  respectively. It is enough to prove that there are constants $c_1$, $c_2$  such that for every $x\in\R^{\pm}$,
\begin{equation}
    \|\varphi_{\pm}(x;\cdot)-\tilde{\varphi}_{\pm}(x;\cdot) \|_{H^1(\R)}\leq c_1 \|q-\tilde q \|_{L^{2,1}(\R)},
\label{36}\end{equation}
and
\begin{equation}
    \|\phi_{\pm}(x;\cdot)-\tilde{\phi}_{\pm}(x;\cdot) \|_{H^1(\R)}\leq c_2 \|q-\tilde q \|_{L^{2,1}(\R)}.
\end{equation}
Moreover, if $q,\tilde q\in H^1(\R)\cap L^{2,1}(\R)$, then there exists constants $c_3$ and $c_4$ such that  for every $x\in\R$,
\begin{equation}
    \| \hat{\varphi}_{\pm}(x;\cdot)-\hat{\tilde{\varphi}}_{\pm}(x;\cdot) \|_{L^2}\leq c_3\|q-\tilde q \|_{H^1\cap L^{2,1}}
  \end{equation}
and
\begin{equation}
    \| \hat{\phi}_{\pm}(x;\cdot)-\hat{\tilde{\phi}}_{\pm}(x;\cdot) \|_{L^2} \leq c_4\|q-\tilde q \|_{H^1\cap L^{2,1}}.
\end{equation}
Again, we only prove the statement (\ref{36}) for $\varphi_-$ . Analogous manipulation lead to other statements.
 Using (\ref{23}), we have
\begin{equation}
    \begin{split}
        \varphi_--\tilde{\varphi}_=&(\varphi_--e_1)-(\tilde{\varphi}_--e_1)\\
        =&(I-K)^{-1}he_2-( I-\tilde{K})^{-1}\tilde{h}e_2\\
        =&( I-K)^{-1}(h-\tilde h)e_2+( I-K)^{-1}(K-\tilde K)( I-\tilde{K})^{-1}\tilde{h}e_2
    \end{split}
\label{39}\end{equation}
where
\begin{equation}
    h(x;k)-\tilde h(x;k)=-\sigma\int_{-\infty}^xe^{2ik(x-y)}(\bar{q}(-y)-\bar{\tilde{q}}(y))\rm dy.
\end{equation}
For the first term of (\ref{39}), it follows from (\ref{17}) that
\begin{equation}
    \sup_{x\in\R} \|h(x;\cdot)-\tilde h(x;\cdot) \|_{L^2(\R)}\leq \sqrt{\pi} \| q-\tilde q\|_{L^{2}(\R)}\leq \sqrt{\pi} \| q-\tilde q\|_{L^{2,1}(\R)}.
\end{equation}
For the second term of (\ref{39}), recall the bounded property (\ref{LxLk}) of the operator K, we have
\begin{equation}
    \| (K-\tilde K)f\|_{L_x^{\infty}L_k^2}\leq \| q-\tilde q\|_{L^{2,1}}\| f\|_{L_x^{\infty}L_k^2},
\end{equation}
where we have used the embedding of $L^{2,1}(\R)$ into $L^1(\R)$.

As a result, we obtain
\begin{equation}
    \|\varphi_{-}(x;\cdot)-\tilde{\varphi}_{-}(x;\cdot) \|_{L^2(\R)}\leq c \|q-\tilde q \|_{L^{2,1}(\R)}.
\end{equation}
Repeating an analogous analysis  for (\ref{331}) for $v=(\partial_k \varphi_-^{(1)}, \partial_k \varphi_-^{(2)}-2ix\varphi_-^{(2)})^t$ , we find
\begin{equation}
    \|\varphi_{-}(x;\cdot)-\tilde{\varphi}_{-}(x;\cdot) \|_{H^1(\R)}\leq c \|q-\tilde q \|_{L^{2,1}(\R)}.
\nonumber\end{equation}
\end{proof}
\subsection{Scattering coefficients and  Lipschitz continuous}
Since both the Jost functions $(\varphi_-(x;k),\phi_-(x;k))$ and $(\varphi_+(x;k),\phi_+(x;k))$ satisfy the first-order linear equation (\ref{spectral}),  there exists a linear dependence
\begin{equation}
    (\varphi_-(x;k)   ,\phi_-(x;k))=(\varphi_+(x;k),\phi_+(x;k))\left(\begin{array}{cc}
    a(k) & c(k)e^{-2ikx}\\
    b(k)e^{2ikx} & d(k)
    \end{array}\right),
\label{linear}
\end{equation}
for every $x\in\R$ and every $k\in\R$.

And given the linear dependence, we have the following properties of the scattering coefficients for $k\in\R$,
\begin{itemize}
    \item The scattering coefficients have the Wronskian's expressions
    \begin{equation}
    \begin{split}
    &a(k)=W(\varphi_-(0;k),\phi_+(0;k)),\\
    &b(k)=W(\varphi_+(0;k),\varphi_-(0;k)),\\
    &d(k)=W(\varphi_+(0;k),\phi_-(0;k)),
    \end{split}\label{ws}
    \end{equation}

    \item The scattering coefficients satisfy the symmetry
     \begin{equation}
    \begin{split}
    &\overline{a(-\bar{k})}=a(k),\\
    &c(k)=-\sigma\overline{b(-\bar{k})},\\
    &\overline{d(-\bar{k})}=d(k),
    \end{split}\label{symm}
    \end{equation}

    \item  The determinant of the scattering matrix is
    \begin{equation}a(k)d(k)+\sigma b(k)\overline{b(-\bar{k})}=1,\label{determi}\end{equation}
    \item The scattering coefficients satisfy the asymptotic
    \begin{equation}
    \begin{split}
    &a(k)\rightarrow 1\quad as\   |k|\rightarrow\infty,\\
    & d(k)\rightarrow 1\quad as\   |k|\rightarrow\infty,\\
    & b(k)\rightarrow 0\quad as\   |k|\rightarrow\infty.
     \end{split}\label{ad}\end{equation}
     \item The scattering coefficients admit the integral expression:
         \begin{equation}
         \begin{split}
         &a(k)=1+\int_{-\infty}^{+\infty} q(y)\varphi_-^{(2)}(y;k)\d y,\\
         &b(k)=-\sigma\int_{-\infty}^{+\infty}e^{-2iky}\bar q(-y)\varphi_-^{(1)}(y;k)\d y,\\
         &d(k)=1-\sigma\int_{-\infty}^{+\infty}\bar q(-y)\phi_-^{(1)}(y;k)\d y.
         \label{binte}\end{split}\end{equation}
\end{itemize}
Define the reflection coefficients
 \begin{equation}
 r_1(k)=\frac{b(k)}{a(k)},\quad r_2(k)=\frac{\overline{b(-k)}}{d(k)},\quad k\in\R.
 \label{r1r2}\end{equation}

\begin{lemma} \label{lemaa4}
If $q\in L^{2,1}(\R)$ and $\| q\|_{L^1(\R)}< 1$, then the function a(k) is continued analytically in $\C^+$, whereas the function d(k) is continued analytically in $\C^-$, in addition,
    $$a(k)-1,\ b(k),\ d(k)-1\in H^1(\R).$$
    Moreover, if $q\in H^{1,1}(\R)$, then
    $$b(k)\in L^{2,1}(\R),\ kb(k)\in L^{\infty}(\R).$$
\end{lemma}
\begin{proof}
By Lemma \ref{2.1}, We  obtained that the Jost functions $\varphi_-(0;k)$ and $\phi_+(0;k)$ are continued analytically in $\C^{+}$, and $\varphi_+(0;k)$ and $\phi_-(0;k)$ are continued analytically in $\C^{-}$. Thus we can derive that $a(k)$ and $d(k)$ are continued analytically in $\C^+$ and $\C^-$ respectively  by using the Wronskian's expressions (\ref{ws}).

The scattering coefficients a(k) can be rewritten as
    \begin{equation}
        \begin{split}
            a(k)-1=&\varphi_-^{(1)}(0;k)\phi_+^{(2)}(0;k)-\varphi_-^{(2)}(0;k)\phi_+^{(1)}(0;k)-1\\
        =&(\varphi_-^{(1)}(0;k)-1)(\phi_+^{(2)}(0;k)-1)+(\phi_+^{(2)}(0;k)-1)\\&+(\varphi_-^{(1)}(0;k)-1)-\varphi_-^{(2)}(0;k)\phi_+^{(1)}(0;k).
        \end{split}
\label{47}    \end{equation}
Since $\varphi_-(0;k)-e_1\in H^1_k(\R), \phi_+(0;k)-e_2\in H^1_k(\R)$ and the space $H^1_k(\R)$ is a Banach algebra, we conclude that $a(k)-1\in H^1_k(\R)$.

Utilizing an analogous method, we have
\begin{equation}
    \begin{split}
     d(k)-1=&\phi_+^{(1)}(0;k)\phi_-^{(2)}(0;k)-\varphi_+^{(2)}(0;k)\phi_-^{(1)}(0;k)-1\\
        =&(\varphi_+^{(1)}(0;k)-1)(\phi_-^{(2)}(0;k)-1)+(\phi_-^{(2)}(0;k)-1)\\&+(\varphi_+^{(1)}(0;k)-1)-\varphi_+^{(2)}(0;k)\phi_-^{(1)}(0;k),
        \end{split}
\label{48}\end{equation}
then we obtain $d(k)-1\in H_k^1(\R)$.

For the scattering coefficients b(k), we find
\begin{equation}
    \begin{split}
        b(k)=&\varphi_+^{(1)}(0;k)\varphi_-^{(2)}(0;k)-\varphi_+^{(2)}(0;k)\varphi_-^{(1)}(0;k)\\
        =&(\varphi_+^{(1)}(0;k)-1)\varphi_-^{(2)}(0;k)+\varphi_-^{(2)}(0;k)\\
        &-\varphi_+^{(2)}(0;k)(\varphi_-^{(1)}-1)-\varphi_+^{(2)}(0;k).
    \end{split}
\label{49}\end{equation}
Again, using the Banach algebra property,  we get $b(k)\in H^1_k(\R)$.

Moreover, if $q\in H^1(\R)\cap L^{2,1}(\R)$, we  rewrite $kb(k)$ as
\begin{equation}
\begin{split}
    kb(k)=&k\varphi_+^{(1)}(0;k)\varphi_-^{(2)}(0;k)-k\varphi_+^{(2)}(0;k)\varphi_-^{(1)}(0;k)\\
    =&\varphi_+^{(1)}(0;k)(k\varphi_-^{(2)}(0;k)-\frac{\sigma}{2i}\bar q(0))-\varphi_-^{(1)}(0;k)(k\varphi_+^{(2)}(0;k)-\frac{\sigma}{2i}\bar q(0))\\
    &+\frac{\sigma}{2i}\bar q(0)(\varphi_+^{(1)}(0;k)-1)-\frac{\sigma}{2i}\bar q(0)(\varphi_-^{(1)}(0;k)-1),
\end{split}\label{50}
\end{equation}
%where we have used the identity observed from (\ref{q12})
%\begin{equation}
%q_+^{(2)}(0)-q_-^{(2)}(0)=0.
%\end{equation}
Recall (\ref{20}) and (\ref{21}),  we find  the all terms in (\ref{50}) are in $L_k^2(\R)$. Therefore, we arrive at  $b(k)\in L^{2,1}_k(\R)$.
Applying the integral expression (\ref{binte}) of $b(k)$ and integrating by part, we have
\begin{equation}
\begin{split}
b(k)=&-\sigma\int_{\R}e^{-2iky}\bar q(-y)(\varphi_-^{(1)}(y;k)-1)\d y-\sigma\int_{\R}e^{-2iky}\bar q(-y)\d y
\\=&-\frac{\sigma}{2ik}\int_{\R}e^{-2iky}\left(\bar q(-y)\partial_y\varphi_-^{(1)}-\partial_y\bar q(-y)(\varphi_-^{(1)}-1) \right)\d y\\
&-\frac{\sigma}{2ik}\int_{\R}e^{-2iky}\partial_y\bar q(-y)\d y.
\end{split}
\end{equation}
Hence, $2i\sigma kb(k)$ also admits an integral expression
\begin{equation}
\begin{split}
2i\sigma kb(k)=&-\int_{\R}e^{-2iky}\left(\bar q(-y)\partial_y\varphi_-^{(1)}-\partial_y\bar q(-y)(\varphi_-^{(1)}-1) \right)\d y\\
&-\int_{\R}e^{-2iky}\partial_y\bar q(-y)\d y.
\end{split}\label{3.62}
\end{equation}
Taking derivative with respect to x gives
$$\partial_x\varphi_-^{(1)}(x;k)=q(x)\varphi_-^{(2)}(x;k).$$
Therefore, we finally obtain
\begin{equation}
\begin{split}
&\|2i\sigma kb(k)\|_{L^{\infty}(\R)}
\\ \leq & \|\varphi_-^{(2)}(x;k) \|_{L_k^{\infty}L_x^{\infty}}\|q \|_{L^2}^2+ \|\varphi_-^{(1)}(x;k) -1 \|_{L_k^{\infty}L_x^{\infty}}\|\partial_xq \|_{L^1} +\|\partial_x q \|_{L^1(\R)}
\end{split}
\end{equation}
Owing to $q\in H^{1,1}(\R)$, we obtain $kb(k)\in L^{\infty}(\R)$.
\end{proof}
\begin{lemma}Let $\| q\|_{L^1(\R)}\leq \frac16$.
The mappings
   \begin{align}
       L^{2,1}(\R)\ni q &\mapsto a(k)-1,\ b(k),\ d(k)-1\in H^1_k(\R),\\
       H^{1,1}(\R)\ni q &\mapsto b(k)\in \H(\R),
   \end{align}
   are Lipschitz continuous  with $$\H(\R)=\left\{r(k)|r(k)\in H^1(\R)\cap L^{2,1}(\R),kr(k)\in L^{\infty}(\R)\right\}.$$
\end{lemma}
\begin{proof}
From the representations (\ref{47}),(\ref{50}) and (\ref{3.62}) and the Lipschitz continuity of  the Jost function $\varphi_{\pm}$ and $\phi_{\pm}$, we can obtain the Lipschitz continuity of the scattering coefficients.
\end{proof}
\begin{lemma}\label{2.77}
 If $q\in L^{2,1}(\R)$ and $\|q\|_{L^1(\R)}<\frac16$, then the spectral problem (\ref{spectral}) admits no eigenvalues or resonances, that is, the scattering coefficients $a(k)$ and $d(k)$ admit no zeros in $\C^+\cup\R$ and $\C^-\cup\R$, respectively.
\end{lemma}
\begin{proof}
Recall that $\varphi_-=e_1+K\varphi_-$ in Lemma \ref{2.1} and the operator $I-K$ is invertible and bounded from (\ref{25}). Using (\ref{LxLk}), we reach that for every $k\in\C^+$,
\begin{equation}
\|\varphi_-(\cdot;k)-e_1\|_{L^{\infty}}=\|(I-K)^{-1}\|\| Ke_1(\cdot;k)\|_{L^{\infty}}\leq e^{2\| q\|_{L^1}}<e^{1/3}.
\label{e2}\end{equation}
Employing (\ref{binte}), we derive for every $k\in\C^+$,
\begin{equation}
|a(k)|\geq 1-\left|       \int_{\R} q(y)\varphi_-^{(2)}(y;k)\d y      \right|> 1-\frac16e^{1/3}>0,
\end{equation}
Due to the continuity of $a(k)$, we obtain $|a(k)|\geq 1-\frac16e^{1/3}>0$ for $k\in\R$. As a result, $a(k)$ admits no zeros in  $\C^+\cup\R$.
Carrying out a similar manipulation for $d(k)$, we see that $d(k)$ admits no zeros in  $\C^-\cup\R$.
\end{proof}
\begin{lemma}
If $q\in L^{2,1}(\R)$ and $\|q\|_{L^1(\R)}<\frac16$, then for every $k\in\R$, we have $|r_{1,2(k)}|<1$.
\end{lemma}
\begin{proof}
Rewrite the expression (\ref{binte}) for $b(k)$ as
\begin{equation*}
b(k)=\int_{\R} e^{-2iky}\bar q(-y)(\varphi_-^{(1)}(y;k)-1)\d y+\int_{\R} e^{-2iky}\bar q(-y)\d y.
\end{equation*}
for every $k\in\R$, applying (\ref{e2}) yields
\begin{equation*}
|b(k)|\leq \|\varphi_-^{(1)}(\cdot;k)-1 \|_{L^{\infty}}\|q \|_{L^1}+\|q \|_{L^1}\leq \frac16(1+e^{1/3}),
\end{equation*}
Due to $|a(k)|\geq 1-\frac16e^{1/3}, k\in\R$, we have
$$|r_1(k)|=\frac{|b(k)|}{|a(k)|}\leq \frac{\frac16(1+e^{1/3})}{1-\frac16e^{1/3}}<1.$$ Similarly, we get $|r_2(k)|< 1$.
\end{proof}
For the reflection coefficients $r_{1,2}(k)$, we have the following results:
\begin{lemma}\label{2.7}
If $q\in L^{2,1}(\R)$ and $\|q\|_{L^1(\R)}<\frac16$, then we have $$r_{1,2}(k)\in H^1(\R).$$
 Moreover, if $q\in H^{1,1}(\R)$, then we have $$r_{1,2}(k)\in L^{2,1}(\R),\ kr_{1,2}(k)\in L^{\infty}(\R),$$
 that is, $r_{1,2}(k)\in\H(\R)$ . As well, the mapping
\begin{equation}
H^{1,1}(\R)\ni q \mapsto (r_1,r_2)\in \H(\R)
\end{equation}
is Lipschitz continuous.
\end{lemma}
\begin{proof}
 Due to
$|a(k)|\geq 1-\frac16e^{1/3}>0$  and $|d(k)|\geq 1-\frac16e^{1/3}>0$ for $k\in\R$,
the property of $r_{1,2}(k)$ follow from the property of $b(k)$.

Let $(r_1,r_2)$ and $(\tilde r_1,\tilde r_2)$ denote the reflection coefficients corresponding to $q$ and $\tilde q$, respectively.
Owing to
\begin{equation}
r_1-\tilde r_1=\frac{b-\tilde b}{a}+\frac{\tilde b((\tilde a-1)-(a-1))}{a\tilde a},
\end{equation}
the Lipschitz continuity of $r_1$ follows from the Lipschitz continuity of $a-1$ and $b$.
\end{proof}

\section{Inverse scattering transform}
\label{sec:section3}
In this section, we will set up a RH problem
 and show  the existence and uniqueness of the solution to the RH problem  for the given data $r_{1,2}(k)\in \H(\R)$ satisfying $|r_{1,2}(k)|<1$.

\subsection{Setup of a RH   problem}
Define the matrix-valued functions
\begin{equation}
    M(x;k)=\left\{\begin{split}
    &\left[ \frac{\varphi_-(x;k)}{a(k)}, \phi_+(x;k) \right],\quad&k\in\C^+\\
    &\left[ \varphi_+(x;k), \frac{\phi_-(x;k)}{d(k)} \right],\quad&k\in\C^-.
    \end{split}\right.
\end{equation}

It follows from the linear dependence (\ref{linear}) and (\ref{symm}) that for $k\in\R$,
\begin{align}
        &\varphi_-=a(k)\varphi_++b(k)e^{+2ikx}\phi_+,\label{s1} \\
        &\phi_-=-\sigma \overline{b(-k)}e^{-2ikx}\varphi_++d(k)\phi_+.\label{s2}
\end{align}
Rewriting (\ref{s2}) as
\begin{equation}
    \phi_+=\frac{\sigma \overline{b(-k)}e^{-2ikx}}{d(k)}\varphi_++\frac{\phi_-}{d(k)}.
    \label{rh1}
\end{equation}
Utilizing (\ref{s1}) and (\ref{rh1}), we have
\begin{equation}
    \begin{split}
        \frac{\varphi_-}{a(k)}=&\varphi_++\frac{b(k)e^{2ikx}}{a(k)}\phi_+ \\
        =&(1+\frac{b(k)}{a(k)}\frac{\sigma\overline{b(-k)}}{d(k)})\varphi_++\frac{b(k)e^{2ikx}}{a(k)}\frac{\phi_-}{d(k)}.
    \end{split}\label{rh2}
\end{equation}
Therefore, for $k\in\R$, we obtain
\begin{equation}
    M_+(x;k)-M_-(x;k)=M_-(x;k)S(x;k),\quad k\in\R,
\label{11}
\end{equation}
where
$$S(x;k)=\left(\begin{array}{cc}
         \sigma r_1(k)r_2(k) & \sigma r_2(k)e^{-2ikx}\\
         r_1(k)e^{2ikx}  &   0
         \end{array}\right).
$$
And the RH problem can be described as

\noindent\textbf{RH problem} Find a $2\times 2$ matrix-valued function M(x;k) with the following properties:
\begin{itemize}
\item $M(x;\cdot):\ \C \backslash\R\rightarrow \C ^{2\times 2}$ is analytic.
\item The limits of $M(x;k)$ as k approaches $\R$ from the upper and lower half-plane exists and are continuous on $\R$, and satisfy
\begin{equation}
    M_+(x;k)-M_-(x;k)=M_-(x;k)S(x;k),\quad k\in\R,
\label{jump}\end{equation}
\item $M(x;k)$ satisfies the asymptotic
$$M(x;k) \rightarrow I,\quad |k|\rightarrow\infty, k\in\C\backslash\R.$$
\end{itemize}
It is convenient to deal with the RH problem to introduce a transformation
$$\Psi_{\pm}(x;k)=M_{\pm}(x;k)-{I},$$
then the jump condition (\ref{11}) can be written as
\begin{align}
    &\Psi_+(x;k)-\Psi_-(x;k)=\Psi_-(x;k)S(x;k)+S(x;k),\quad k\in \R,\label{12}\\
    &\Psi_{\pm}(x;k)\rightarrow 0,\quad |k|\rightarrow \infty, k\in\C\backslash\R.
\end{align}

\subsection{Solvability of the RH problem}

Before looking for the solution to the RH problem, we show some preliminary knowledge that will be used in the subsequent section. And this results have been given in the previous work.\cite{19}

For any function $h\in L^p(\R)$ with $1\leq p<\infty$, the Cauchy operator is defined as
\begin{equation}
    \mathcal C (h)(k):=\frac{1}{2\pi i}\int_{\R} \frac{h(s)}{s-k}\rm ds,\quad k\in \C\backslash \R,
\end{equation}
and the Plemelj projection is given by
\begin{equation}
    \mathcal P^{\pm}(h)(k)=\lim_{\epsilon\downarrow 0}\frac{1}{2\pi i} \int_{\R} \frac{h(s)}{s-(k\pm \epsilon i)}\rm ds,\quad k\in \R.
\end{equation}
\begin{proposition}(see \cite{18}) For every $h\in L^p(\R),\ 1\leq p<\infty $, the Cauchy operator $\mathcal{C}(h)$ is analytic off the real line, decays to zero as $|k|\rightarrow\infty$, and approaches to $\mathcal P^{\pm}(h)$ almost everywhere, when a point $k\in \C^{\pm}$ approaches to a point on the real axis by any non-tangential contour from $\C^{\pm}$. If $1< p<\infty$, then there exists a positive constant $C_p$(with $C_{p=2}=1$) such that \begin{equation}
    \| \mathcal P^{\pm}(h)\|_{L^p}\leq C_p \| h\|_{L^p},
\label{56}\end{equation}
If $h\in L^1(\R)$, then the Cauchy operator admits the following asymptotic limit in either $\C^+$ or $\C^-$:
\begin{equation}
    \lim_{|k|\rightarrow \infty} k\mathcal C(h)(k)=-\frac{1}{2\pi i}\int_R h(s)\d s.
\label{57}\end{equation}
\end{proposition}
\begin{lemma}\label{5.1} For every $r_{1,2}(k)\in H^1(\R)$ satisfying $|r_{1,2}(k)|<1$.
%satisfying
%\begin{equation}
%    \begin{split}
%        &1+   \sigma \rm{Re}(r_1(k)r_2(k)) -\frac{1}{4}   |r_1(k)+ \sigma \bar r_2(k)|^2  >0.
%    \end{split}\label{small}
%\end{equation}
there exist positive constants $c_-$ and $c_+$ such that for every $x\in\R$ and every column-vector $g\in \C^2$
\begin{equation}
    \rm Re\ g^*( I+S(x;k))g\geq c_-g^*g,\quad k\in\R
\label{60}\end{equation}
and
\begin{equation}
    \|(I+S(x;k))g\|\leq c_+\|g \|,\quad k\in\R,
\label{61}
\end{equation}
where the asterisk denotes the Hermite conjugate.
\end{lemma}

%\begin{remark}  The constraint (\ref{small}) will be used to obtain a unique solution to the RH problem. Moreover,
%this  assumption yields
%$$1+   \sigma \rm{Re}(r_1(k)r_2(k))>0 $$
%which ensures    that   the scattering coefficients $a(k)$ and $d(k)$
% admits no zeros in $\C$. Indeed,
%it follows from (\ref{determi}) that
%\begin{equation}
%\frac{1}{a(k)d(k)}=1+\sigma r_1(k)r_2(k),
%\end{equation}
%hence, we have
%$$1+\sigma\rm{Re}(r_1r_2)=\frac{\rm{Re}(ad)}{\rm{Re}^2(ad)+\rm{Im}^2(ad)}.$$
%Then $\rm{Re}(ad)>0$ follows from $1+\sigma\rm{Re}(r_1r_2)>0$. As a result, we obtain $|a(k)d(k)|>0, k\in\C$. Therefore, $a(k)$ and $d(k)$ admits no zeros in $\C$.
%\end{remark}
\begin{proof}
The original scattering matrix $S(x;k)$ is not Hermitian due to the fact there is no symmetry between $a(k)$ and $d(k)$. Hence  it is difficult to use the theory of Zhou \cite{20} to obtain a unique solution to the RH problem.
Therefore, we define Hermitian part of $S(x;k)$ by
\begin{equation}
\begin{split}
    S_H(x;k)&=\frac 12 S(x;k)+\frac12 S^*(x;k)\\
    &=\left( \begin{array}{cc}
    \sigma \rm Re(r_1r_2) & \frac12(\bar r_1+\sigma r_2)e^{-2ikx}\\
    \frac 12(r_1+\sigma \bar{r}_2)e^{2ikx} & 0
    \end{array}\right),
\end{split}
\end{equation}
Since $|r_{1,2}(k)|<1$, we have the 2-order principle minor of the matrix $I+S_H$
$$1+\sigma\rm{Re}(r_1r_2)-\frac14|r_1+\sigma\bar r_2|^2=1-\frac14|r_1-\sigma\bar r_2|^2>0,$$
further, we have the 1-order  principle minor $1+\sigma\rm{Re}(r_1r_2)>0$. As a result, the matrix $I+S_H$ is positive definite.

In view of the algebra theory, for a Hermitian matrix, there exists an unitary matrix A such that $$A^*(I+S_H)A={\rm diag}(
\mu_+, \mu_- ),$$
where $\mu_{\pm}$ are the eigenvalues of the matrix $I+S_H$
\begin{equation}
    \mu_{\pm}(k)=\frac{2+\sigma \rm{Re}(r_1r_2)\pm\sqrt{\rm{Re}^2(r_1r_2)+|r_1+\sigma \bar{r}_2|^2}}{2}.
    \nonumber
\end{equation}
    Note $\mu_+(k)>\mu_-(k)>0$ since $I+S_H$ is positive definite. It follows from $r_{1,2}(k)\rightarrow 0$  that $\mu_-(k)\rightarrow 1$ as $|k|\rightarrow\infty,\ k\in\R$. Due to  $r_{1,2}(k)\in H^1(\R)$, there exists a constant $c_->0$ such that $\mu_->c_->0$.

Consequently, for every $g\in\C^2$, let $\tilde g=A^{-1}g$, we have
\begin{equation}
    c_-g^*g<\mu_-g^*g\leq\rm Re\ g^*( I+S(x;k))g=g^*(I+S_H)g=\tilde g^*A^*(I+S_H)A\tilde g.
\end{equation}
%where
%\begin{equation}
%    \begin{split}
%        &c_-=\frac{1}{2+\sup_{k\in\R} \big( \rm{Re} (r_1r_2)\big)}>0,
%    \end{split}
%    \nonumber
%\end{equation}
%where we have used the identity
%\begin{equation}
%\frac{2+\alpha-\sqrt{\alpha^2 +4\alpha}}{2}\cdot(2+\alpha+\sqrt{\alpha^2 +4\alpha})=2
%\end{equation}
%with $\alpha=\sigma\rm{Re}(r_1r_2)$.

Finally, calculating componentwise gives that
\begin{equation}
    \begin{split}
       \|(I+S(x;k))g \|^2\leq & 2(1+|r_1 |^2+|r_2 |^2)\| g \|^2\\
       &+2 \rm{Re}\{ ((1+r_1r_2)\bar r_2+r_1)e^{2ikx}g^{(1)}\overline{g^{(2)}}\}\\
       \leq&\left(2(1+|r_1 |^2+|r_2 |^2)(1+| r_2|)+|r_1 |\right)\| g \|^2,
    \end{split}
\end{equation}
here the norm for a 2-component vector $f$ is $\|f\|^2=| f^{(1)}|^2+| f^{(2)}|^2$.
Let
$$c_+=\sup_{k\in\R}\left(\sqrt{2(1+|r_1 |^2+|r_2 |^2)(1+| r_2|)+|r_1 |}\right)<+\infty,$$
we obtain the bound (\ref{61}).
\end{proof}
\begin{lemma}\label{5.2} For every $r_{1,2}(k)\in H^1(\R)$ satisfying $|r_{1,2}(k)|<1$ and every $F(k)\in L^2_k(\R)$, there exists a unique solution $\Psi(k)\in L_k^2(\R)$ of the equation
\begin{equation}
(\rm I-\P_S^-)\Psi(k)=F(k),\ k\in\R,
\end{equation}
where $\P_S^-\Psi(k)=\P^-(\Psi S)$.
\end{lemma}
\begin{proof}
Since $\rm I-\P_S^-$ is a Fredholm operator of the index zero,\cite{21}\cite{22} by Fredholm's alternative theorem, there exists a unique solution of the equation $(\rm I-\P_S^-)\Psi(k)=F(k)$ if and only if the zero solution of the equation $(\rm I-\P_S^-)g=0  $ is the unique solution in $L^2_k(\R)$.

Assume that there exists a function $g(k)\in L^2_k(\R)$ and $g(k)\neq 0$ such that $(\rm I-\P_S^-)g=0 $.
Define two analytic functions in $\mathbb C\backslash \mathbb R$
\begin{equation}
    g_1(k)=\mathcal C(gS)(k), \quad g_2(k)=\mathcal C(gS)^*(k).
\end{equation}
 The functions $g_1(k)$ and $g_2(k)$ are well-defined due to $S(k)\in L_k^2(\mathbb R)\cap L_k^{\infty}(\mathbb R)$.

 We integrate the function $g_1(k)g_2(k)$ along the semi-circle of radius R centered at zero in $\mathbb C^+$, it follows from Cauchy theorem that
\begin{equation}
    \oint g_1(k)g_2(k)\d k=0.
\end{equation}
Since $g(k)S(k)\in L_k^1(\mathbb R)$, by using (\ref{57}), we have $g_{1,2}=\mathcal O(k^{-1}),\ |k|\rightarrow \infty$. Thus, the integral on the arc tends to zero as the radius tends to infinity. Therefore, we obtain
\begin{equation}
    \begin{split}
        0&=\int_{\R} g_1(k)g_2(k) \d k =\int_{\R} \mathcal P^+(gS)[\mathcal P^-(gS)]^*\rm dk\\
        &=\int_{\R}[P^-(gS)+gS][P^-(gS)]^*\rm dk,
    \end{split}
\end{equation}
Utilizing the assumption $\P^-(gS)=g$, we have
\begin{equation}
\int_{\R} g(I+S)g^*\d k=0.
\end{equation}
 By Lemma 5.1, we have $\rm{Re}\ g(I+S)g^*>c_-g^*g$ with $c_-$ is a positive constant. Hence the function g(k) has to be zero function. This contradicts to the assumption $g\neq 0$. As a result, $g=0$ is a unique solution to the equation $(I-\P^-_S)g=0$ in $L^2_k(\R)$. Finally, there exists a unique solution to the equation $(I-\P_S^-)\Psi(k)=F(k)$.
\end{proof}
\begin{lemma}\label{th5.1}
 For $r_{1,2}(k)\in H^1(\rm{R})\cap L^{2,1}(\rm{R})$ satisfying $|r_{1,2}(k)|<1$ and for every $x\in\R$,
there exists a unique solution $\Psi_{\pm}(x;k)\in L^2_k(\R)$ of the problem
$$\Psi_+(x;k)-\Psi_-(x;k)=\Psi_-S(x;k)+S(x;k),\quad k\in\R $$
Moreover, $\Psi_{\pm}(x;k)$ are analytic functions for $k\in\C^{\pm}$.
\end{lemma}
\begin{proof}
 Owing to $S(x;k)\in L^2_k(\R)$, we have $\P^-(S)(k)\in L^2_k(\R)$ by (\ref{56}), hence,  we infer from Lemma \ref{5.2} that there exists a unique solution $\Psi_-(x;k)\in L^2_k(\R)$ to the  problem for every $x\in\R$
 \begin{equation}
      \Psi_-(x;k)=\P^-(\Psi_-(x;k)S(x;k)+S(x;k)),\quad k\in\R.\label{psi-}
 \end{equation}
 Then we define a function $\Psi_+(x;k)$ by
 \begin{equation}
     \Psi_+(x;k)=\P^+(\Psi_-(x;k)S(x;k)+S(x;k)),\quad k\in\R,\label{psi+}
 \end{equation}
 And analytic extensions of $\Psi_{\pm}(x;k)$ to $k\in \C^{\pm}$ are defined by Cauchy integrals
 \begin{equation}
     \Psi_{\pm}(x;k)=\mathcal C(\Psi_-(x;k)S(x;k)+S(x;k)),\quad k\in\C^{\pm}.
 \end{equation}
 Finally we obtain the solution $\Psi_{\pm}(x;k)\in L^2_k(\R)$ of the  problem. Moreover, given the property of the Cauchy operator and projection operator, the solutions $\Psi_{\pm}(x;k)$ are analytic functions for $k\in\C^{\pm}$.
\end{proof}
\begin{lemma}
For every $r_{1,2}(k)\in H^1(\R)$ satisfying $|r_{1,2}(k)|<1$, the operator $(I-\P_S^-)^{-1}$ is bounded from $L_k^2(\R)$ to $L_k^2(\R)$, and there exists a constant $c$ that only depends on $\| r_{1,2}(k)\|_{L_k^{\infty}}$ such that
\begin{equation}
\|(I-\P_S^-)^{-1} f \|_{L_k^2}\leq c\| f\|_{L_k^2}
\label{92}\end{equation}\label{3.3}
\end{lemma}
\begin{proof}
For every $f(k)\in L_k^2(\R)$, it follows from Lemma \ref{5.2} that there exists a solution $\Psi(k)\in L_k^2(\R)$ to $(I-\P_S^-)\Psi(k)=f(k)$. Note that $\P^+-\P^{-}=I$, we decompose the function $\Psi(k)$ into $\Psi=\Psi_+-\Psi_-$ with
\begin{equation}
\Psi_--\P^-(\Psi_-S)=\P^-(f),\quad \Psi_+-\P^-(\Psi_+S)=\P^+(f).
\label{69}
\end{equation}
Since $\P^{\pm}(f)\in L_k^2(\R)$, it follows from Lemma \ref{5.2} that there exists unique solutions for the equation (\ref{69}),  hence, the decomposition is unique.
Therefore, we only need the estimates of $\Psi_{\pm}$ in $L^2_k(\R)$.

To deal with $\Psi_-$, define two analytic functions in $\C\backslash\R$
$$g_1(k)=\mathcal C(\Psi_-S)(k),\quad g_2(k)=\mathcal C(\Psi_-S+f)^*(k),$$
Analogous manipulation as the proof of Lemma \ref{5.2}, we integrate on the semi-circle in the upper half-plane and have
\begin{equation}
    \oint g_1(k)g_2(k)\d k=0.
\end{equation}
Since $g_1(k)=\mathcal O(k^{-1})$ and $g_2(k)\rightarrow 0$ as $|k|\rightarrow\infty$, we have
\begin{equation}
    \begin{split}
        0=&\int_{\R}\P^+(\Psi_-S)[\P^-(\Psi_-S+f)]^* \d k\\
        =&\int_{\R}(\P^-(\Psi_-S)+\Psi_- S)[\P^-(\Psi_-S+f)]^* \d k\\
        =&\int_{\R}(\Psi_--\P^-(f)+\Psi_-S)\Psi_-^* \d k,
    \end{split}
\end{equation}
By the bound (\ref{60}) and the H\"{o}lder inequality, there exists a positive constant $C_-$ such that
\begin{equation}
    c_-\| \Psi_-\|_{L^2}^2\leq \rm{Re} \int_{\R} \Psi_-(I+S)\Psi_-^* \d k=\rm{Re}\int_{\R} \P^-(f)\Psi_-^*\d k\leq \| f\|_{L^2}\| \Psi_-\|_{L^2},
\end{equation}
this completes the estimates of $\Psi_-$:
\begin{equation}
 \|(I-\P_S^{-})^{-1}\P^-f \|_{L^2_k}\leq c_-^{-1}\| f\|_{L^2_k}.
    \label{f1}
\end{equation}
To deal with $\Psi_+$, define two functions in $\C\backslash\R$
$$g_1(k)=\mathcal C(\Psi_+S)(k),\quad g_2(k)=\mathcal C(\Psi_+S+f)^*(k).$$
Performing the similar procedure leads to
\begin{equation}
    \begin{split}
        0=&\oint g_1(k)g_2(k) \d k\\
        =&\int_{\R} \P^-(\Psi_+S)[\P^+(\Psi_+S+f)]^* \d k\\
        =&\int_{\R}[\Psi_+-\P^+(f)][\Psi_+(I+S)]^* \d k,
        \end{split}
\end{equation}
where we have used (\ref{69}).
By the bounds (\ref{60}) and (\ref{61}) in Lemma \ref{5.1}, there are positive constants $c_-$ and $c_+$ such that
\begin{equation}
    c_-\| \Psi_+\|_{L^2}^2\leq \rm{Re} \int_{\R} \Psi_+(I+S)^*\Psi_+^* \d k=\rm{Re} \int_{\R}\P^+(f)(I+S)^*\Psi_+^* \d k\leq c_+\|f \|_{L^2}\| \Psi_+\|_{L^2},
    \end{equation}
which means
\begin{equation}
    \| (I-\P_S^-)^{-1}\P^+f\|_{L^2_k}\leq c_-^{-1}c_+\| f\|_{L^2_k}.\label{f2}
\end{equation}
Combining (\ref{f1}) and (\ref{f2}), we obtain
$$\|(I-\P_S^-)^{-1} f \|_{L_k^2}\leq c\| f\|_{L_k^2}.$$
\end{proof}

\subsection{Estimate on solution to  the RH problem}

Next, we come back to the original RH problem about $M(x;k)$. Denote the functions $M_{\pm}$ column-wise
$$M_{\pm}(x;k)=[\mu_{\pm}(x;k), \nu_{\pm}(x;k)].$$

 We write column-wise the functions $\Psi_{\pm}$
$$\Psi_{\pm}(x;k)=[\mu_{\pm}(x;k)-e_1, \nu_{\pm}(x;k)-e_2 ].$$
According to (\ref{psi-}) and (\ref{psi+}),  we have the first column
\begin{equation}
    \mu_{\pm}(x;k)-e_1=\P^{\pm}(\psi_- S+S)^{(1)}(x;k)=\P^{\pm}(M_-S)^{(1)}(x;k),\quad k\in{\R},
\end{equation}
and the second column
\begin{equation}
    \nu_{\pm}(x;k)-e_2=\P^{\pm}(\psi_- S+S)^{(2)}(x;k)=\P^{\pm}(M_-S)^{(2)}(x;k),\quad k\in{\R}.
\end{equation}
Thus we obtain
\begin{equation}
M_{\pm}(x;k)={I}+\P^{\pm}(M_-(x;\cdot)S(x;\cdot))(k),\quad k\in\R,
\label{pleme}\end{equation}
furthermore,
\begin{equation}
M_{\pm}(x;k)={I}+\mathcal C (M_-(x;\cdot)S(x;\cdot))(k),\quad k\in{\C}^{\pm},
\label{79}\end{equation}
By Lemma \ref{th5.1}, there exists a unique solution $M_{\pm}(x;k)$ to (\ref{pleme}) and (\ref{79}).
\begin{lemma}\label{3.4} Let $r_{1,2}\in H^1(\R)\cap L^{2,1}(\R)$ satisfy $|r_{1,2}(k)|<1$, then there exists a constant $c$ only depending on $\| r_{1,2}\|_{L^{\infty}}$ such that for every $x\in\R$, the solution $M_{\pm}(x;k)$ satisfies
\begin{equation}
    \| M_{\pm}(x;\cdot)-I\|_{L^2}\leq c(\|r_1 \|_{L^2}+\|r_2 \|_{L^2} ).\label{3.42}
\end{equation}
\end{lemma}
\begin{proof}
Due to $r_{1,2}(k)\in H_1(\R)\cap L^{2,1}(\R)$, we find  $r_{1,2}(k)\in L^2(\R)\cap L^{\infty}(\R)$ and $S(x;k)\in L^2_k(\R)$. Moreover, there exists a constant $c_1$ only depending on $\|r_{1,2}\|_{L^{\infty}}$ such that
\begin{equation}
    \|S(x;k) \|_{L_k^2}\leq c_1(\| r_1\|_{L^2}+\| r_2\|_{L^2}),
\end{equation}
By Lemma \ref{3.3} and the bound (\ref{56}), we obtain
\begin{equation}
    \|M_{\pm}-I \|=\| \psi_{\pm}\|_{L^2}\leq c_2\|\P^-S \|_{L^2}\leq c_2\| S\|_{L^2}\\
    \leq c(\| r_1\|_{L^2}+\| r_2\|_{L^2}),
\end{equation}
where we have used the equation $(I-\P_S^-)\Psi_{\pm}=\P^-S$, and $c_2$ and $c$ are  constants only depending on $\| r_{1,2}\|_{L^{\infty}}$.
\end{proof}
\begin{proposition}(see \cite{18})  For every $x_0\in\R^+$ and every $r_{1,2}\in H^1(\R)$, we have
\begin{align}
    &\sup_{x\in (x_0,+\infty)} \| \langle x \rangle\P^+(r_2(k)e^{-2ikx})\|_{L^2_k} \leq \| r_2\|_{H^1},\label{83}\\
    &\sup_{x\in (x_0,+\infty)} \| \langle x \rangle \P^-(r_1(k)e^{2ikx})\|_{L^2_k} \leq \| r_1\|_{H^1},\label{109}
\end{align}
where $\langle x \rangle=(1+x^2)^{1/2}$, moreover, we have
\begin{align}
    &\sup_{x\in \R} \|  \P^+(r_2(k)e^{-2ikx})\|_{L^{\infty}_k} \leq \frac{1}{\sqrt{2}}\| r_2\|_{H^1},\label{85}\\
    &\sup_{x\in \R} \|  \P^-(r_1(k)e^{2ikx})\|_{L^{\infty}_k} \leq \frac{1}{\sqrt{2}}\| r_1\|_{H^1},\label{886}
\end{align}
furthermore, if $r_{1,2}\in L^{2,1}(\R)$, then we have
\begin{align}
    &\sup_{x\in\R} \|\P^+(kr_2(k)e^{-2ikx}) \|_{L^2_k}\leq \| kr_2(k)\|_{L^2_k},\label{87}\\
    &\sup_{x\in\R} \|\P^-(k{r}_1(k)e^{2ikx}) \|_{L^2_k}\leq \| kr_1(k)\|_{L^2_k}\label{872}.
\end{align}
%\begin{proof}
%By an analogous analysis in \cite{18}, applying the Fubini's theorem and Jordan theorem, we get
%\begin{equation}
%\P^+(r_2(k)e^{2ikx})=\int_{-2x}^{+\infty}\hat r_2(z)e^{ik(z+2x)}\d z=\pi\int_{-2x}^{+\infty}|\hat r_2(z)|^2\d z,
%\end{equation}
%If $-\frac{z}{2}<x<0$, we have for every $x_0\in\R^-$,
%\begin{equation}
%\sup_{x\in(-\infty,x_0)}\|\langle x\rangle\P^+(r_2(k)e^{2ikx})\|_{L^2_k(\R)}\leq \sqrt{2\pi}\|\hat r_2\|_{L^{2,1}(\R)}\leq \| r_2\|_{H^1(\R)}.
%\end{equation}
%\end{proof}
\end{proposition}
In order to obtain the estimates on the vector columns $\mu_--e_1$ and $\nu_+-e_2$ that will be needed in the subsequent section, we rewrite
the functions $\mu_-(x;k)-e_1$ and $\nu_+(x;k)-e_2$ by (\ref{pleme}),
\begin{equation}
\mu_-(x;k)-e_1=\P^-(M_-S(x;\cdot))^{(1)}(k)=\P^-(r_1(k)e^{2ikx}\nu_+(x;k))(k), k\in\R
\label{89}\end{equation}
and
\begin{equation}
\nu_+(x;k)-e_2=\P^+(M_-S(x;\cdot))^{(2)}(k)=\sigma \P^+(r_2(k)e^{-2ikx}\mu_-(x;k))(k), k\in\R,
\label{91}
\end{equation}
where we have used the fact
\begin{equation}
    \begin{split}
        M_-S=&[\mu_-,\nu_-]\left( \begin{array}{cc}
        \sigma r_1r_2 & \sigma r_2e^{-2ikx}\\
        r_1e^{2ikx} & 0
        \end{array}
        \right)\\
        =&[\sigma r_1r_2\mu_-+r_1e^{2ikx}\nu_-,\sigma r_2 e^{-2ikx}\mu_-]
    \end{split}
\end{equation}
and the identity
\begin{equation}
    \begin{split}
        &\nu_+=\sigma r_2 e^{-2ikx}\mu_-+\nu_-
    \end{split}\label{11117}
\end{equation}
 follows from (\ref{jump})
 \begin{equation}
 [\mu_+,\nu_+]=[\mu_-,\nu_-]\left( \begin{array}{cc}
 1+\sigma r_1r_2 & \sigma r_2e^{-2ikx}\\
 r_1e^{2ikx} & 1
 \end{array}\right).
 \end{equation}

Introduce a function
\begin{equation}
    M(x;k)=[\mu_-(x;k)-e_1, \nu_+(x;k)-e_2]
\end{equation}
then we have
\begin{equation}
    \begin{split}
        M-\P^+(MS_+)-\P^-(MS_-)=F,
    \end{split}
\label{1120}\end{equation}
where
\begin{equation}
    \begin{split}
        &F(x;k)=[\P^-(r_1e^{2ikx})e_2,\P^+(\sigma r_2e^{-2ikx})e_1],\\
        &S_+(x;k)=\left(\begin{array}{cc}
        0 & \sigma r_2e^{-2ikx}\\
        0 & 0
        \end{array}\right),\quad S_-(x;k)=\left(\begin{array}{cc}
        0 & 0\\
        r_1e^{2ikx} & 0
        \end{array}\right).
    \end{split}
\nonumber\end{equation}
\begin{lemma} For every $x_0\in\R^+$ and every $r_{1,2}\in H^1(\R)$, the solution of (\ref{89}) and (\ref{91}) satisfies
\begin{align}
    &\sup_{x\in (x_0,+\infty)} \| \langle x \rangle \mu_-^{(2)}(x;k)\|_{L^2_k}\leq c\| r_1\|_{H^1},\label{xmu}\\
    &\sup_{x\in (x_0,+\infty)} \| \langle x \rangle \nu_+^{(1)}(x;k)\|_{L^2_k}\leq c\| r_2\|_{H^1},\label{96}
\end{align}
where $c$ is a constant that only depends on $\|r_{1,2} \|_{L^{\infty}}$.
If $r_{1,2}\in \H(\R)$, then we have
\begin{align}
&\sup_{x\in (x_0,+\infty)} \| \partial_x \mu_-^{(2)}(x;k)\|_{L^2_k}\leq c(\|r_1 \|_{\H}+\|r_2 \|_{\H})\label{982}\\
&\sup_{x\in (x_0,+\infty)} \| \partial_x \nu_+^{(1)}(x;k)\|_{L^2_k}\leq c(\|r_1 \|_{\H}+\|r_2 \|_{\H}),\label{98}
\end{align}
where c is another  constant that depends on $\|r_{1,2} \|_{L^{\infty}}$ and $\|kr_{1,2}(k) \|_{L^{\infty}}$.
\end{lemma}
\begin{proof}
Recall  $\P^+-\P^-=I$ and
$$S_++S_-=(I-S_+)S,$$
Eq.(\ref{1120}) can be rewritten as
\begin{equation}
    G-\P^-(GS)=F
\label{1125}\end{equation}
with $G=M(I-S_+)$. The matrix G(x;k) is written component-wise as
\begin{equation}
    G(x;k)=\left( \begin{array}{cc}
     \mu_-^{(1)}(x;k)-1    &  \nu_+^{(1)}-\sigma r_2e^{-2ikx}(\mu_-^{(1)}(x;k)-1)\\
    \mu_-^{(2)}(x;k)     &  \nu_+^{(2)}-1-\sigma r_2e^{-2ikx}\mu_-^{(2)}(x;k)
    \end{array}\right).
\end{equation}
Comparing The second row of F(x;k) with G(x;k) and considering the bound (\ref{92}) and (\ref{109}), we have
\begin{equation}
\begin{split}
    &\sup_{x\in (-\infty,x_0)}\|\langle x \rangle \mu_-^{(2)} \|_{L^2_k}\leq c\sup_{x\in (-\infty,x_0)}\|\langle x\rangle \P^-(r_1e^{2ikx}) \|_{L^2_K}\leq c\| r_1\|_{H^1},\\
   &\|\nu_+^{(2)}-1-\sigma r_2e^{2ikx}\mu_-^{(2)}(x;k) \|_{L^2_k}\leq c\|\P^-(r_1e^{2ikx}) \|_{L^2_k},
\end{split}\label{1277}
\end{equation}
this completes the proof of (\ref{xmu}).

Similarly, Comparing the first row of F(x;k) and G(x;k) yields
\begin{equation}
    \begin{split}
        &\|\mu_-^{(1)}(x;k)-1 \|_{L^2_k}\leq c\|\P^+(r_2e^{-2ikx}) \|_{L^2_k}, \\
        &\|\nu_+^{(1)}(x;k)-\sigma r_2e^{-2ikx}(\mu_-^{(1)}(x;k)-1) \|_{L^2_k}\leq c\|\P^+(r_2e^{-2ikx}) \|_{L^2_k}.
    \end{split}\label{1288}
\end{equation}
Because of $r_2\in L^{\infty}(\R)$ and the triangle inequality, we have
\begin{equation}
    \sup_{x\in (-\infty,x_0)}\|\langle x \rangle\nu_+^{(1)}(x;k) \|_{L^2_k}\leq c^{'}\sup_{x\in (-\infty,x_0)}\|\langle x \rangle\P^+(r_2e^{-2ikx}) \|\leq c\|r_2 \|_{H^1},
\end{equation}
this completes the proof of (\ref{96}).

Taking derivative in x of (\ref{1120}), we obtain
\begin{equation}
    \partial_x M-\P^+(\partial_x M)S_+-\P^-(\partial_x M)S_-=\tilde F
\end{equation}
with
\begin{equation}
    \begin{split}
        \tilde F=&\partial_x F+\P^+M\partial_xS_++\P^-M\partial_xS_- \\
        =&2i[e_2\P^-(-kr_1(k)e^{2ikx}),e_1\P^+(\sigma kr_2(k)e^{-2ikx})] \\
        &+2i\left(\begin{array}{cc}
        \P^-(-kr_1(k)e^{2ikx}\nu_+^{(1)}(x;k)) &\sigma\P^+(kr_2(k)e^{-2ikx}(\mu_-^{(1)}(x;k)-1)) \\
        \P^-(-kr_1(k)e^{2ikx}(\nu_+^{(2)}(x;k)-1)) & \sigma\P^+(kr_2(k)e^{-2ikx}\mu_-^{(2)}(x;k))
        \end{array}\right).
    \end{split}    \nonumber
\end{equation}
%Rewriting $kr_{1,2}(k)$ as
%\begin{equation}
%        kr_{1,2}^2(k)=\int_{-\infty}^k r_{1,2}^2(s)+2sr_1(s)\partial_sr_{1,2}(s)\d s,\quad k\in\R,
%\end{equation}
%we have $kr_{1,2}(k)\in L^{\infty}(\R)$ thanks to $r_{1,2}\in H^1(\R)\cap L^{2,1}(\R)$.
According to the estimates (\ref{1277})- (\ref{1288}), we obtain  $\mu_-(x;k)-e_1\in L_x^{\infty}((-\infty,x_0);L^2_k(\R))$ and $\nu_+(x;k)-e_2\in L_x^{\infty}((-\infty,x_0);L^2_k(\R))$ . On account of
$kr_{1,2}(k)\in L^{\infty}(\R)$ , we conclude that   $\tilde F$ belongs to $L_x^{\infty}((-\infty,x_0);L^2_k(\R))$.

Then we repeat the analysis about (\ref{1125}) and  give the estimates (\ref{982}) and (\ref{98}).
\end{proof}

\section{Reconstruction and estimates of the potential}
\label{sec:section4}

Comparing the 2-element of the limits (\ref{9}) leads to
\begin{equation}
    \bar{q}(-x)=2i\sigma \lim_{|k|\rightarrow\infty} k\varphi_{\pm}^{(2)}(x;k).
\label{99}\end{equation}
As well, Comparing the 2-element of the limits (\ref{10}) leads to
\begin{equation}
    {q}(x)=-2i\lim_{|k|\rightarrow\infty} k\phi_{\pm}^{(1)}(x;k).
\label{100}\end{equation}
It follows from (\ref{99}) and (\ref{79}) that
\begin{equation}
    \bar{q}(-x)=2i\sigma\lim_{|k|\rightarrow\infty}k\mathcal C((M_-S)_{21} )\label{136}
\end{equation}
If $r_{1,2}\in H^1(\R)\cap L^{2,1}(\R)$, then $S(x;\cdot)\in L^1(\R)\cap L^2(\R)$. And the estimate (\ref{3.42}) implies that $M_-(x;\cdot)-I\in L^2(\R)$, hence, we arrive at $M_-S=(M_--I)S+S\in L^1(\R)$. Subsequently, Applying (\ref{57}) to (\ref{136}), we obtain
\begin{equation}
    \begin{split}
        \bar{q}(- x)=&-\frac{\sigma}{\pi}\int_{\R} r_1(k)e^{2ikx}( \nu_-^{(2)}(x;k)+\sigma  r_2(k) e^{-2ikx}\mu_-^{(2)}(x;k) )\d k\\
        =&-\frac{\sigma}{\pi}\int_{\R}r_1(k)e^{2ikx}\nu_+^{(2)}(x;k)\d k,
    \end{split}\label{134}
\end{equation}
where we have used the equation $\nu_+^{(2)}=\sigma r_2(k)e^{-2ikx}\mu_-^{(2)}+\nu_-^{(2)}$ due to (\ref{11117}).

Performing the same manipulation
for (\ref{100})
yields
\begin{equation}
    {q}(x)=\frac{\sigma}{\pi}\int_{\R} r_2(k)e^{-2ikx} \mu_-^{(1)}(x;k) \d k.
\label{103}\end{equation}

\begin{lemma} \label{3.6}
Let $r_{1,2}(k)\in \H(\R)$ satisfying $|r_{1,2}(k)|<1$, then
$q\in H^{1,1}(\R^+)$, moreover,
\begin{equation}
    \| q\|_{H^{1,1}(\R^+)}\leq c(  \| r_1\|_{\H(\R)}+\|r_2\|_{\H(\R)}),
\label{104}\end{equation}
where c is a constant that  depends on $\| r_{1,2}\|_{L^{\infty}}$ and $\|kr_{1,2}\|_{L^{\infty}}$.
\end{lemma}

\begin{proof}
We rewrite (\ref{103} ) for ${q}(x)$ as
\begin{equation}
    {q}(x)=\frac{\sigma}{\pi}\int_{\R} r_2(k)e^{-2ikx} \d k+\frac{\sigma}{\pi} \int_{\R} r_2(k)e^{-2ikx} ( \mu_-{(1)}(x;k)-1 )\d k.
\label{105}\end{equation}
Recall the results from the Fourier theory. For a function $r(k)\in L^2(\R)$, by Parseval's equation, we have
\begin{equation}
    \| r\|_{L^2}=\|\hat{r} \|_{L^2},
\end{equation}
where the function $\hat{r}$ denotes the Fourier transform with the definition $$\hat{r}(x)=\frac{1}{2\pi}\int_{\R}r(k)e^{-ikx}dz.$$
Since $r_2\in H^1(\R)$, the first term of (\ref{105}) belongs to $L^{2,1}(\R)$  due to the property $\widehat{\partial_k r(k)}=x\hat{r}(x)$.

Let
\begin{equation}
    \uppercase\expandafter{\romannumeral1}(x)=\int_{\R} r_2(k)e^{-2ikx} ( \mu_-^{(1)}(x;k)-1 ) \d k.
\end{equation}
Substituting (\ref{91}) into the above expression and applying  Fubini's theorem yields
\begin{equation}
    \begin{split}
       \uppercase\expandafter{\romannumeral1}(x)=&\sigma\int_{\R} r_2(k)e^{-2ikx} \lim_{\epsilon\rightarrow 0} \frac{1}{2\pi i}\int_{\R}\frac{r_1(s)e^{2isx}\nu_+^{(1)}(s)}{s-(k-i\epsilon)}\d s\d k\\
       =&-\sigma\int_{\R}r_1(s) e^{2isx} \nu_+^{(1)}(s)\lim_{\epsilon\rightarrow 0}\frac{1}{2\pi i}\int_{\R}\frac{r_2(k)e^{-2ikx}}{k-(s+i\epsilon)} \d k\d s\\
       =& -\sigma\int_{ \R} r_1( k)e^{2ikx}\nu_+^{(1)}(k) \P^+(r_2(k)e^{-2ikx})(k)\d k,
    \end{split}
\end{equation}
thus for every $x_0\in \R^+$, utilizing the H\"{o}lder's inequality and  the estimates (\ref{109}) and (\ref{xmu}), we find
\begin{equation}
    \begin{split}
      &  \sup_{x\in (x_0,+\infty)}\left| \langle x\rangle^2\uppercase\expandafter{\romannumeral1}(x) \right|\leq  \| r_1\|_{L^{\infty}} \sup_{x\in (x_0,+\infty)} \|\langle x\rangle\nu_+^{(1)}(x;k) \|_{L_k^2}\\
        &\qquad\times\sup_{x\in (x_0,+\infty)} \|\langle x\rangle\P^+(r_2(k)e^{-2ikx}) \|_{L_k^2}
        \leq   c_1\|r_2 \|_{H^1}^2,
    \end{split}
\end{equation}
where $c_1$ is a constant only depends on $\|r_{1,2} \|_{L^{\infty}}$,
hence we obtain
\begin{equation}
    \begin{split}
        \|\langle x\rangle I(x) \|_{L^2(\R^+)}
        %=&\left(\int_{\R^+}\frac{1}{\langle x\rangle^2}\langle x\rangle^4 |\uppercase\expandafter{\romannumeral1}(x) |^2\d x \right)^{\frac 12}\\
        \leq & c_1 \sup_{x\in\R^-} | \langle x\rangle^2 I(x)  |
        \leq   c_1\|r_2 \|_{H^1}^2.
    \end{split}
\end{equation}
 Combining the results of the two terms of (\ref{105}) leads to
 \begin{equation}
     \|q(x)\|_{L^{2,1}( \R^+)}\leq c_1^{}(1+\| r_2\|_{H^1})\| r_2\|_{H^1}.
\label{111} \end{equation}
 This completes  the proof of $q\in L^{2,1}(\R^+)$.
 By Fourier theory, the derivative of the first term of (\ref{105}) belongs to $L^2(\R)$. For the second term $I(x)$, we differentiate $I(x)$ in $x$ and obtain
 \begin{equation}
     \begin{split}
         I'(x)=&\frac{\partial}{\partial x}\int_{\R} r_2(k)e^{-2ikx}(\mu_-^{(1)}(x;k)-1) \d k\\
         =&-2i\int_{\R}kr_2(k)e^{-2ikx}(\mu_-^{(1)}(x;k)-1 )\d k
         \\&+\int_{\R} r_2(k)e^{-2ikx}\frac{\partial \mu_-^{(1)}(x;k)}{\partial x} \d k\\
         =&2i\int_{\R} r_1(k) e^{2ikx}\nu_+^{(1)}(x;k)\P^+(kr_2(k)e^{-2ikx})(k)\d k\\
         &-2i\int_{\R} kr_1(k)e^{2ikx}\nu_+^{(1)}(x;k)\P^+(r_2(k)e^{-2ikx})(k)\d k\\
         &-\int_{\R} r_1(k) e^{2ikx} \partial_x \nu_+^{(1)}(x;k)\P^+(r_2(k)e^{-2ikx})(k)\d k,
     \end{split}
 \end{equation}
 where we have used the equation (\ref{91}) and Fubini's theorem.

 Utilizing the estimates (\ref{83}), (\ref{85}), (\ref{87}), (\ref{96}) and (\ref{98}), we find that for every $x_0\in\R^+ $,
 \begin{equation}
     \begin{split}
         \sup_{x\in (x_0,+\infty)}|\langle x\rangle I'(x)  | \leq &2\| r_1\|_{L^{\infty}} \sup\|\langle x\rangle\nu_+^{(1)}(x;k) \|_{L^2_k}\sup \| \P^+(kr_2(k)e^{-2ikx})\|_{L^2_k}\\
         &+2\| kr_1\|_{L^2}\sup\|\langle x\rangle\nu_+^{(1)}(x;k) \|_{L^2_k}\sup \| \P^+(r_2(k)e^{-2ikx})\|_{L^{\infty}_k}\\
         &+\| r_1\|_{L^{\infty}}\sup\|\partial_x\nu_+^{(1)}(x;k) \|_{L^2_k}\sup \|\langle x\rangle \P^+(r_2(k)e^{-2ikx})\|_{L^2_k}\\
         \leq & c_2\|r_1 \|_{\H}\|r_2 \|_{\H}(\|r_1 \|_{\H}+\|r_2 \|_{\H}),
     \end{split}
\nonumber \end{equation}
which implies that
\begin{equation}
\begin{split}
\| \langle x\rangle I'(x) \|_{L^2(\R^+)}
%=&\left(\int_{\R^+}\frac{1}{\langle x \rangle^2}|\langle x \rangle \uppercase\expandafter{\romannumeral1}^{'}(x) |^2 \d x\right)^{\frac 12}\\
\leq & c_2\|r_1 \|_{\H}\|r_2 \|_{\H}(\|r_1 \|_{\H}+\|r_2 \|_{\H}),
\label{1113}\end{split}
\end{equation}
with $c_2$ is a constant that depends on $\| r_{1,2}\|_{L^{\infty}}$ and $\|kr_{1,2}\|_{L^{\infty}}$,
hence $I'(x)\in L^{2,1}(\R^+)$, then $q\in H^{1,1}(\R^+)$. The estimate (\ref{104}) can be obtained from
(\ref{111}) and (\ref{1113}) with a constant c depending on $\| r_{1,2}\|_{L^{\infty}}$ and $\|kr_{1,2}\|_{L^{\infty}}$.
\end{proof}
\begin{lemma} \label{3.7}Let $r_{1,2}(k)\in \H(\R)$ satisfying $|r_{1,2}(k)|<1$, then the mapping
\begin{equation}
\H(\R)\ni (r_1,r_2)\mapsto q\in H^{1,1}(\R^+)
\end{equation}
is Lipschitz continuous.
\end{lemma}
\begin{proof}
Let $(r_1,r_2),(\tilde r_1,\tilde r_2)\in \H(\R)$. Let the functions q and $\tilde q$ are the corresponding potentials respectively.  We will show that there exists a constant that  depends on $\| r_{1,2}\|_{L^{\infty}}$ and $\|kr_{1,2}\|_{L^{\infty}}$ such that
\begin{equation}
    \|q-\tilde q \|_{H^{1,1}(\R^+)}\leq c(\|r_1-\tilde r_1 \|_{\H(\R)}+ \|r_2-\tilde r_2 \|_{\H(\R)}). \label{wiw}
\end{equation}
From (\ref{103}) and (\ref{105}), we have
\begin{equation}
\begin{split}
    q-\tilde q=&\frac{\sigma}{\pi}\int_{\R}(r_2-\tilde r_2)e^{-2ikx}\d k+\int_{\R}\frac{\sigma}{\pi}(r_2-\tilde r_2)e^{-2ikx}(\mu_-^{(1)}(x;k)-1)\d k\\
    &+\int_{\R}\frac{\sigma}{\pi}\tilde r_2(k)e^{-2ikx}(\mu_-^{(1)}(x;k)-\tilde\mu_-^{(1)}(x;k))\d k.
\end{split}
\end{equation}
Repeating the analysis in the proof of Lemma \ref{3.6}, we obtain the Lipschitz continuity of q.
\end{proof}
\begin{lemma}\label{4.8}
Let $r_{1,2}(k)\in \H(\R)$ satisfying $|r_{1,2}(k)|<1$, then
$q\in H^{1,1}(\R^-)$, moreover,
\begin{equation}
    \| q\|_{H^{1,1}(\R^-)}\leq c(  \| r_1\|_{\H(\R)}+\|r_2\|_{\H(\R)}    ),
\end{equation}
where c is a constant that depends on $\| r_{1,2}\|_{L^{\infty}}$ and $\|kr_{1,2}\|_{L^{\infty}}$.
\end{lemma}
\begin{proof}
We rewrite (\ref{134}) for $\bar q(-x)$ as
\begin{equation}
    \bar q(-x)=-\frac{\sigma}{\pi}\int_{\R} r_1(k)e^{2ikx}  \d k-\frac{\sigma}{\pi}\int_{\R} r_1(k)e^{2ikx} (\nu_+^{(2)}(x;k)-1) \d k.
\end{equation}
Let
\begin{equation}
    \hat r_1(-x)=\int_{\R} r_1(k)e^{-2ik(-x)}\d k.
\end{equation}
According to the Fourier theory, we have
\begin{equation}
    -x\hat r_1(-x)=\widehat{\partial_kr_1(k)}(x)
\end{equation}
and
\begin{equation}
    \|x\hat r_1(-x)\|_{L^2(\R)}=\| \partial_k r_1(k)\|_{L^2(\R)}.
\end{equation}
Since $r_1\in H^1(\R)$,  we obtain
\begin{equation}
    \|\hat r_1(-x) \|_{L^{2,1}(\R)}\leq \| \hat r_1(-x)\|_{L^2(\R)}+\| x\hat r_1(-x)\|_{L^2(\R)}= \| r_1\|_{H^1(\R)}.
\end{equation}
Let
$$\uppercase\expandafter{\romannumeral2}(x)=-\frac{\sigma}{\pi}\int_{\R}r_1(k)e^{2ikx}(\nu_+^{(2)}(x;k)-1)\d k.$$
Repeating the analysis in the proof of Lemma \ref{3.6}, we obtain $q\in H^{1,1}(\R^-)$ and
\begin{equation}
    \| q\|_{H^{1,1}(\R^-)}\leq c(  \| r_1\|_{\H(\R)}+\|r_2\|_{\H(\R)}   ).
\end{equation}
\end{proof}
By similar procedure as Lemma \ref{3.7}, we have the following results:
\begin{lemma}\label{4.9}
 Let $r_{1,2}(k)\in \H(\R)$ satisfying $|r_{1,2}(k)|<1$, then the mapping
\begin{equation}
\H(\R)\ni (r_1,r_2)\mapsto q\in H^{1,1}(\R^-)
\end{equation}
is Lipschitz continuous.
\end{lemma}

We summary the results in Lemma  \ref{3.6} to Lemma  \ref{4.9}, we have the following proposition

\begin{proposition} \label{prop4.1}
Let $r_{1,2}(k)\in \H(\R)$ satisfying $|r_{1,2}(k)|<1$, then
$q (x) \in H^{1,1}(\R )$ and
\begin{equation}
    \| q\|_{H^{1,1}(\R^+)}\leq c(  \| r_1\|_{\H(\R)}+\|r_2\|_{\H(\R)}).
\label{104}\end{equation}
More the mapping
\begin{equation}
\H(\R)\ni (r_1,r_2)\mapsto q\in H^{1,1}(\R )
\end{equation}
is Lipschitz continuous.
\end{proposition}

\section{ Existence of global solutions}
\label{sec:section5}

\subsection{Time evolution of scattering data }

In  Section \ref{sec:section2}  to  Section \ref{sec:section5},   for  initial data   $q(0,x)\in H^{1,1}(\mathbb{R})$,  we   consider spatial spectral problem (\ref{spectral}) and
 obtain its   a unique solution
\begin{align}
  & \varphi_\pm(0,x; k)\rightarrow e_1,  \  \  \phi_\pm(0,x; k)\rightarrow e_2,   \ \ x \rightarrow  \pm\infty,  \label{pw1}
\end{align}
which cannot satisfy time spectral problem (\ref{3}) since they  are short of a function about  time $t$.
For  every $t\in [0,T]$, we define  the normalized Jost functions  of the Lax pair (\ref{spectral}) and (\ref{3})
\begin{align}
&\varphi_{\pm}(t,x;k)=\varphi_{\pm}(0,x;k)e^{-2ik^2t}, \label{time1}\\
 &\phi_{\pm}(t,x;k)=\phi_{\pm}(0,x;k)e^{2ik^2t},\label{time2}
\end{align}
with the potential $q(0,x)\in H^{1,1}(\R)$.
It follows that for every $t\in[0,T]$, we have
$$\varphi_{\pm}(t,x;k)\rightarrow e^{-2ik^2t}e_1 \quad x\rightarrow\pm\infty,$$
$$\phi_{\pm}(t,x;k)\rightarrow e^{2ik^2t}e_2 \quad x\rightarrow\pm\infty.$$
Repeating the analysis as  the proof of Lemma \ref{2.1}, we prove that there exist unique solutions of the Volttera's integral equation for  Jost functions $\varphi_{\pm}(t,x;k)$ and $\phi_{\pm}(t,x;k)$, and the  Jost functions $\varphi_{\pm}(t,x;k)$ and $\phi_{\pm}(t,x;k)$ admit the same
 analytic property as $\varphi_{\pm}(0,x;k)$ and $\phi_{\pm}(0,x;k)$.

As well,   for every $(x,t)\in\R\times\R^+$ and every $k\in\R$,  the Jost functions $\varphi_\pm(t,x;k)$ and $ \phi_\pm(t,x;k)$  should satisfy  the  scattering  relation
\begin{equation}
\begin{split}
    &\varphi_-(t,x;k)=a(t;k)\varphi_+(t,x;k)+b(t;k)e^{2ikx }\phi_+(t,x;k),\\
    &\phi_-(t,x;k)=c(t;k)e^{-2ikx }\varphi_+(t,x;k)+d(t;k)\phi_+(t,x;k).
\end{split}
\nonumber\end{equation}
By Crammer's law and  evolution relation  (\ref{time1})-(\ref{time2}),  we obtain evolution   of   the scattering coefficients
\begin{equation}
\begin{split}
&a(t;k)=W(\varphi_-(0,0;k)e^{-2ik^2t},\phi_+(0,0;k)e^{ 2ik^2t})=a(0;k),\\
&b(t;k)=W(\varphi_+(0,0;k)e^{-2ik^2t},\varphi_-(0,0;k)e^{-2ik^2t}) = b(0;k) e^{-4ik^2t},\\
&d(t;k)=W(\phi_+(0,0;k)e^{2ik^2t}, \phi_-(0,0;k)e^{-2ik^2t})=d(0;k).
\end{split}\nonumber
\end{equation}

Direct calculation shows that   reflection scattering coefficients are given by
\begin{align}
& r_1(t;k)=\frac{b(t;k) }{a(t;k)}=\frac{b(0;k)}{a(0;k)}e^{-4ik^2t}=r_{1}(0;k)e^{-4ik^2t},\label{eeew}\\
& r_2(t;k)=\frac{b(t;-k) }{d(t;k)}=\frac{\overline{b(-k)}}{d(k)}e^{4ik^2t}=r_{2}(0;k)e^{4ik^2t},
\end{align}
where $r_{1,2}(0;k)$ are initial reflection data found from the initial data $q(0,x)$.

\begin{proposition}
If $r_{1,2}(0;k)\in \H(\R)$, then for a fixed $T>0$ and every $t\in[0,T]$, we have $r_{1,2}(t;k)\in \H(\R).$
\end{proposition}

\begin{proof}
From (\ref{eeew}), we obtain
\begin{equation}
\| r_{1,2}(t;\cdot) \|_{L^{2,1}(\R)}=\| r_{1,2}(0;\cdot) \|_{L^{2,1}(\R)},
\end{equation}
and using Lemma \ref{2.7},
\begin{equation}
\| kr_{1,2}(t;\cdot) \|_{L^{\infty}(\R)}=\| kr_{1,2}(0;\cdot) \|_{L^{\infty}(\R)}.
\end{equation}
For every $t\in[0,T]$, we have
\begin{equation}
\begin{split}
\| \| \partial_kr_{1}(t;\cdot) \|_{L^{2}(\R)}=&\|\partial_k r_1(0;\cdot)-8iktr_1(t;k) \|_{L^2(\R)}\\
\leq & \|\partial_k r_1(0;\cdot) \|_{L^2(\R)}+8T\|r_1(0;\cdot) \|_{L^{2,1}},
\end{split}\label{6.3}
\end{equation}
\begin{equation}
\begin{split}
\| \| \partial_kr_{2}(t;\cdot) \|_{L^{2}(\R)}=&\|\partial_k r_2(0;\cdot)+8iktr_2(t;k) \|_{L^2(\R)}\\
\leq & \|\partial_k r_2(0;\cdot) \|_{L^2(\R)}+8T\|r_2(0;\cdot) \|_{L^{2,1}}.
\end{split}\label{6.4}
\end{equation}
Thus we infer  that $r_{1,2}(t;\cdot)\in \H(\R)$ for every $t\in[0,T]$ as $r_{1,2}(0;\cdot)\in \H(\R)$.
\end{proof}

\subsection{Local and global solution }
In this section, we will prove the existence of the local solution and global solution to the Cauchy problem (\ref{nnls})-(\ref{initial}).
The scheme behind the proof of the existence of local and global solution can be described as below

\begin{center}
\begin{tikzpicture}
\node at (-4.0, 1.0 ){\fontsize{8pt}{8pt} $q(0,x)\in H^{1,1}(\mathbb{R}) $};
\draw [->] (-2.7,1.0)--(2.8,1.0);
\node at (4.0, 1.0 ){\fontsize{8pt}{8pt} $r_{1,2}(0;k)\in \mathcal{H}(\mathbb{R}) $};

\node at (-4.0, -1.0 ){\fontsize{8pt}{8pt} $q(t,x) \in C([0,\infty),H^{1,1}(\R)) $};
\draw [<-] (-2,-1.0)--(3,-1.0);
\node at (4.2, -1.0 ){\fontsize{8pt}{8pt} $r_{1,2}(t;k) \in \mathcal{H}(\mathbb{R}) $};

\draw [-] (4,0.8)--(4,-0.6);\draw [->] (4,-0.6)--(4,-0.8);
\draw [dashed] (-4,0.8)--(-4,-0.6);\draw [->] (-4,-0.6)--(-4,-0.8);

\node at (0, 1.2 ){\footnotesize Lipschitz};
\node at (0, -1.2 ){\footnotesize Lipschitz};
\end{tikzpicture}
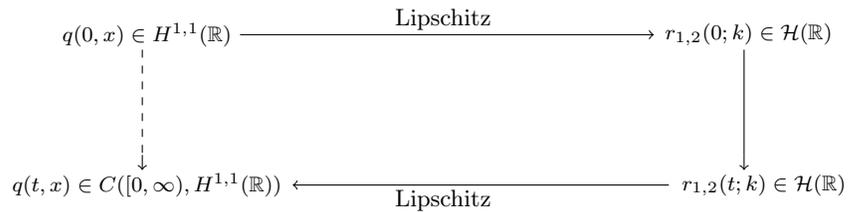
\center\footnotesize{\textbf{Fig.  1} The general scheme  }
\end{center}
\begin{theorem}    Let   the initial data  $q_0(x)\in H^{1,1}(\R)$  nd $\|q_0 \|_{L^1(\R)}<\frac16$,
then there exists a unique local solution  to the Cauchy problem (\ref{nnls})-(\ref{initial})
\begin{align}
q(t,x)\in C([0,T],H^{1,1}(\R)),\ t\in[0,T].
\end{align}
Furthermore, the map
\begin{equation}
  H^{1,1}(\mathbb{R}) \ni q_{0}(x) \mapsto q(t,x) \in C\left([0, T], H^{1,1}(\mathbb{R})\right)
\end{equation}
is Lipschitz continuous.
\end{theorem}
\begin{proof}
The constraint satisfying $|r_{1,2}(k)|<1$  remains valid for $r_{1,2}(t;k)$ for every $t\in[0,T]$.
Performing a similar analysis as Lemma \ref{3.6}-\ref{4.9},  we can establish a RH problem for $r_{1,2}(t,x;k)$ for every $t\in[0,T]$ and address the existence and uniqueness of the solution to the RH problem. Further, the potential q(t,x) can be recovered from the reflection coefficients $r_{1,2}(t,x;k)$. Moreover, the potential q(t,x) belongs to $H^{1,1}(\R)$ for every $t\in[0,T]$ and is Lipschitz continuous of $r_{1,2}(t;k)$. Thus we have
\begin{equation}
\begin{split}
    \| q(t;\cdot) \|_{H^{1,1}} &\leq c_1(\|r_1(t;\cdot)\|_{\H}+\|r_2(t;\cdot)\|_{\H })\\
    &\leq c_2(r_1(0;\cdot)\|_{\H}+\|r_2(0;\cdot)\|_{\H })   \leq c_3 \|q_0 \|_{H^{1,1}},
\end{split}\label{158}
\end{equation}
where the positive constants $c_1$,$c_2$ and $c_3$ depends on $T, \| r_{1,2}\|_{L^{\infty}}$ and $  \|kr_{1,2}\|_{L^{\infty}} $.

Next we show $u(t,x )$ is continuous  with respect to   every $t\in[0,T]$ under $H^{1,1}(\mathbb{R})$ norm.
Let $t\in \in[0,T]$ and $|\Delta t|<1$ such that $t+\Delta t \in [0,T]$,  then with the Lipschitz continuity from $u(t,x)$ to $r_{1,2}(t;z)$ in Proposition \ref{prop4.1}, we
have
\begin{equation}
\begin{aligned}\nonumber
&\|u(t+\Delta t, x  )-u( t, t+\Delta t  )\|_{H^3\cap H^{2,1}} \nonumber\\
&\leq    c(  \| r_1(t+\Delta t, x  )-r_1( t,x  ) \|_{\H(\R)}+ \| r_2(t+\Delta t, x  )-r_2( t, x  ) \|_{\H(\R))})\\
&\leq c|\Delta t| ( \|r_{1 }(z)\|_{\H }+  \|r_{2 }(z)\|_{\H }) \leq c |\Delta t| \rightarrow 0, \ \Delta t \to 0,
\end{aligned}
\end{equation}
which together with the estimate (\ref{158}) implies that there exists a unique local solution $q(t,x)\in C([0,T],H^{1,1}(\R)),\ t\in[0,T]$ to the Cauchy problem (\ref{nnls})-(\ref{initial} and the map
 $$H^{1,1}(\mathbb{R}) \ni q_{0}(x) \mapsto q(t,x) \in C\left([0, T], H^{1,1}(\mathbb{R})\right)$$
 is Lipschitz continuous.
 \end{proof}
%\begin{theorem}
%Let the assumption 1  holds. For every initial data $q_0(x)\in H^{1,1}(\R)$ and $q_0(x)\rightarrow 0$ as $x\rightarrow \pm\infty$, there exists a unique local solution $q(t,x)\in C([0,T],H^{1,1}(\R))$ of the Cauchy problem (\ref{nnls}) with the initial data (\ref{initial}) for $t\in[0,T]$.
%\end{theorem}
The following theorem shows that there exists a global solution in $H^{1,1}(\R)$:
\begin{theorem}    Let   the initial data  $q_0(x)\in H^{1,1}(\R)$  and $\|q_0 \|_{L^1(\R)}<\frac16$,
then there exists a unique global solution  to the Cauchy problem (\ref{nnls})-(\ref{initial})
\begin{align}
q(t,x)\in C([0,\infty),H^{1,1}(\R)).
\end{align}
Furthermore, the map
\begin{equation}
  H^{1,1}(\mathbb{R}) \ni q_{0}(x) \mapsto q(t,x) \in C\left([0, \infty);  H^{1,1}(\mathbb{R})\right)
\end{equation}
is Lipschitz continuous.
\end{theorem}
\begin{proof}
Suppose the maximal time in which the local solution exists is $T_{max}$.

If $T_{max}=\infty$, then the local solution is global.

If the local solution exists in the closed interval $[0,T_{max}]$, we can use $q(T_{max},x)\in H^{1,1}(\R)$ as a new initial data. By a similar analysis as the previous sections, there exists a positive constant $T_1$ such that the solution $q(t,x)\in C([T_{max},T_{max}+T_1], H^{1,1}(\R))$ exists. This contradicts with the maximal time assumption.

If the local solution exists in the open interval $[0,T_{max})$. According to (\ref{158}), we have
\begin{equation}
\|q(t,x) \|_{H^{1,1}(\R)}\leq c_3(T_{max})\|q_0 \|_{H^{1,1}(\R)},\quad t\in[0,T_{max}).
\end{equation}
 Due to the continuity of $q(t,x)$ to the time $t$, the limit of $q(t,x)$ as $t$ approaches to $T_{max} $ exists. Let $ q_{max}(x):=\lim_{t\rightarrow T_{max}} q(t,x)$. Taking the limit by $t\rightarrow T_{max}$ in  (\ref{158}), we have
 \begin{equation}
 \|q_{max} \|_{H^{1,1}(\R)}\leq  c_3(T_{max})\|q_0 \|_{H^{1,1}(\R)},
 \end{equation}
 which implies that we can extend the local solution $q(t,x)\in C([0,T_{max}),H^{1,1}(\R))$ to $q(t,x)\in C([0,T_{max}],H^{1,1}(\R))$, this contradicts with the premise that $[0,T_{max})$ is the maximal open interval.
%Assume $q\in C([0,T],H^{1,1}(\R))$ blows up in  a finite time, that is, there exists a constant $T_{max}$ such that
%\begin{equation}
%    \lim_{t\rightarrow T_{max}} \|q(t;\cdot) \|_{H^{1,1}(\R)}=\infty.
%    \label{limitinf}
%\end{equation}
%According to (\ref{6.3})(\ref{6.4}) and (\ref{158}), we know that $c_3$ remains finite for every $T>0$. This contradicts on the assumption (\ref{limitinf}). Therefore, we can extend the local solution $q\in C([0,T],H^{1,1}(\R))$ to the whole time line $q\in C([0,T],H^{1,1}(\R))$.
\end{proof}

\end{document}